\documentclass[oneeqnum,onethmnum,onefignum]{siamltex}

\usepackage{amsmath}
\usepackage{amsfonts}
\usepackage[dvips]{graphicx}
\usepackage[tight]{subfigure}
\usepackage{tikz}
\usepackage{hyperref}

\newcommand\rset{\mathbb{R}}
\newcommand\tset{\mathbb{T}}
\newcommand\zset{\mathbb{Z}}
\newcommand\diff{\mathrm{d}}

\title{Fast transport optimization for Monge costs on~the~circle%
  \thanks{Supported by the French National Research Agency (project
    ANR-07-BLAN-0235 OTARIE,
    \texttt{http://www.mccme.ru/\~{}ansobol/otarie/}). %
    The hospitality of UMR~6202 CNRS ``Laboratoire Cassiop{\'e}e''
    (Observatoire de la C{\^o}te d'Azur) is gratefully acknowledged.}}

\author{Julie Delon\thanks{LTCI CNRS, T{\'e}l{\'e}com ParisTech.} %
  \and Julien Salomon\thanks{Universit{\'e}
    Paris-Dauphine and CEREMADE.} %
  \and Andrei Sobolevski\thanks{Institute for Information Transmission
    Problems and UMI 2615 CNRS ``Laboratoire J.-V.~Poncelet,'' Moscow;
    partially supported by the Ministry of National Education of
    France and by the Russian Fund for Basic Research (project RFBR
    07--01--92217-CNRSL-a).}}

\begin{document}

\maketitle

\begin{abstract}
Consider the problem of optimally matching two measures on the
circle, or equivalently two periodic measures on~$\rset$, and
suppose the cost $c(x, y)$ of matching two points $x$,~$y$ satisfies
the Monge condition: $c(x_1, y_1) + c(x_2, y_2) < c(x_1, y_2) +
c(x_2, y_1)$ whenever $x_1 < x_2$ and $y_1 < y_2$. %
We introduce a notion of \protect\emph{locally optimal} transport
plan, motivated by the weak KAM (Aubry--Mather) theory, and show
that all locally optimal transport plans are conjugate to shifts and
that the cost of a locally optimal transport plan is a convex
function of a shift parameter. %

This theory is applied to a transportation problem arising in image
processing: for two sets of point masses on the circle, both of which
have the same total mass, find an optimal transport plan with respect
to a given cost function~$c$ satisfying the Monge condition. %
In the circular case the sorting strategy fails to provide a unique
candidate solution and a naive approach requires a quadratic number of
operations. %
For the case of $N$ real-valued point masses we present an
$O(N|\log\epsilon|)$ algorithm that approximates the optimal cost
within~$\epsilon$; when all masses are integer multiples of~$1/M$, the
algorithm gives an exact solution in $O(N\log M)$ operations.
\end{abstract}



\pagestyle{myheadings}
\thispagestyle{plain}
\markboth{J.~DELON, J.~SALOMON AND A.~SOBOLEVSKI}{FAST TRANSPORT
  OPTIMIZATION ON THE CIRCLE}

\section{Introduction}
\label{sec:intro}

The transport optimization problem, introduced by G.~Monge in 1781 and
shown by L.~Kantorovich in 1942 to be an instance of linear
programming, is a convex optimization problem with strong geometric
features. %
A typical example is minimization of mean-square displacement between
two given finite marginal measures supported on convex compacta in
Euclidean space: in this case a solution is defined by gradient of a
convex function that satisfies a suitable Monge--Amp{\`e}re
equation. %
Various generalizations of this result and rich bibliographies can be
found, e.g., in the survey \cite{Gangbo.W:1996} or the recent
monograph \cite{Villani.C:2009}.

Further constraints on the two marginals or their supports may furnish
the problem with useful additional convex structure. %
One way of making this statement quantitative is to consider the
algorithmic complexity of the corresponding numerical transport
optimization schemes. %
In particular when the two measures live on segments of straight
lines, the optimal map is monotone and may be found by sorting, which
takes $O(n\log n)$ operations when the data come in the form of
discrete $n$-point histograms. %
If the input data are already sorted, this count falls to $O(n)$. %

The optimal transport is well understood also when the marginals live
on a compact Riemannian manifold \cite{Feldman.M:2002}; the existence
and characterization of optimal map in the case of a flat torus and
quadratic cost have been established a decade ago
\cite{Cordero-Erausquin.D:1999}. %
However, the algorithmics of even the simplest setting of the unit
circle is no longer trivial, because the support of the measures is
now \emph{oriented} rather than \emph{ordered}. %
A naive approach would require solving the problem for each of $n$
different alignments of two $n$-point histograms, thereby involving
$O(n^2)$ operations. %

In this paper we present an efficient algorithm of transport
optimization on the circle, which is based on a novel analogy with the
Aubry--Mather (weak KAM) theory in Lagrangian dynamics (see, e.g.,
\cite{Aubry.S:1983,Fathi.A:2009,Knill.O:2003}). %
The key step is to lift the transport problem to the universal cover
of the unit circle, rendering the marginals periodic and the cost of
transport infinite. %
However, it still makes sense to look for those transport maps whose
cost cannot be decreased by any \emph{local} modification. %
Different locally optimal maps, which cannot be deformed into each
other by any local rearrangement, form a family parameterized with an
analogue of the rotation number in the Aubry--Mather theory. %
One can introduce a counterpart of the average Lagrangian, or Mather's
$\alpha$ function in the Aubry--Mather theory, which turns out to be
efficiently computable. %
As we show below, its minimization provides an efficient algorithm of
transport optimization on the circle. %
The class of cost functions for which this theory works includes all
costs with the Monge property, such as the quadratic cost or costs
generated by natural Lagrangians with time-periodic potentials
\cite{Bernard.P:2007c,Knill.O:2003}. %

Note that the problem of optimally matching circular distributions
appears in a variety of applications. %
Important examples are provided by image processing and computer
vision: image matching techniques for retrieval, classification, or
stitching purposes \cite{Zhang.J:2007,Brown.M:2005} are often based on
matching or clustering ``descriptors'' of local features
\cite{Lowe.D:2004}, which typically consist of one or multiple
histograms of gradient orientation. %
Similar issues arise in object pose estimation and pattern recognition
\cite{Lowe.D:2004,Gangbo.W:2000}. %
Circular distributions also appear in a quite different context of
analysis of color images, where hue is parameterized by polar angle. %
In all these applications, matching techniques must be robust to data
quantization and noise and computationally effective, which is
especially important with modern large image collections. %
These requirements are satisfied by the optimal value of a transport
cost for a suitable cost function. %

This paper is organized as follows. %
In \S\ref{sec:informal} we give a specific but nontechnical overview
of our results. %
After the basic definitions are given in \S\ref{sec:preliminaries},
including that of locally optimal transport plans, in
\S\ref{sec:conjugate} we give an explicit description of the family of
locally optimal transport plans: they are conjugate, in measure
theoretic sense, to rotations of the unit circle (or equivalently to
shifts of its universal cover). %
This result is in direct analogy with conjugacy to rotations in the
one-dimensional Aubry--Mather theory \cite{Aubry.S:1983}. %
As shown in \S\ref{sec:average-cost}, the average cost of a locally
optimal transport plan is a convex function of the shift parameter. %
Moreover, the values of this function and its derivative are
efficiently computable when the marginal measures are discrete, which
enables us to present in \S\ref{sec:algor-transp-optim} a fast
algorithm for transport optimization on the circle. %
The same section contains results of a few numerical experiments. %
Finally a review of related work in the computer science literature is
given in \S\ref{sec:disc-concl}. %

\section{Informal overview}
\label{sec:informal}

For two probability measures $\hat\mu_0$,~$\hat\mu_1$ on the unit
circle $\tset = \rset/\zset$ and a given cost $\hat c(\hat x, \hat y)$
of transporting a unit mass from $\hat x$ to~$\hat y$ in~$\tset$, the
\emph{transport cost} is defined as the $\inf$ of the quantity
\begin{equation}
  \label{eq:1}
  \hat I(\gamma) =  \iint_{\tset\times\tset} \hat c(\hat x, \hat y)\,
  \gamma(\diff\hat x\times\diff\hat y).
\end{equation}
over the set of all \emph{couplings} $\gamma$ of the probability
measures~$\hat\mu_0$,~$\hat\mu_1$ (i.e., all measures on
$\tset\times\tset$ with marginals $\hat\mu_0$,~$\hat\mu_1$). %
These couplings are usually called \emph{transport plans}. %

Suppose that the cost function $\hat c(\cdot, \cdot)$ on
$\tset\times\tset$ is determined via the relation $\hat c(\hat x, \hat
y) = \inf c(x, y)$ by a function $c(\cdot, \cdot)$ on
$\rset\times\rset$ satisfying the condition $c(x + 1, y + 1) = c(x,
y)$ for all~$x$,~$y$; here $\inf$ is taken over all $x$,~$y$ whose
projections to the unit circle coincide with $\hat x$,~$\hat y$. %
We lift the measures $\hat\mu_0$ and~$\hat\mu_1$ to~$\rset$, obtaining
periodic locally finite measures $\mu_0$,~$\mu_1$, and redefine
$\gamma$ to be their coupling on $\rset\times\rset$. %
It is then convenient to replace the problem of minimizing the
integral~\eqref{eq:1} with ``minimization'' of an integral
\begin{equation}
  \label{eq:2}
  I(\gamma) = \iint_{\rset\times\rset} c(x, y)\, \gamma(\diff
  x\times\diff y).
\end{equation}
Although the latter integral is infinite, it still makes sense to look
for transport plans $\gamma$ minimizing $I$ with respect to
\emph{local} modifications, i.e., to require that for any compactly
supported signed measure $\delta$ of zero mass and finite total
variation, the difference $I(\gamma + \delta) - I(\gamma)$, which is
defined by a finite integral, be nonnegative. %
These \emph{locally optimal} transport plans are the main object of
this paper.

Assume that the cost function $c(x, y)$ satisfies the \emph{Monge
  condition} (alternatively known as the continuous Monge property,
see~\cite{Aggarwal.A:1992,Burkard.R:1996}):
\begin{equation}
  \label{eq:3}
  c(x_1, y_1) + c(x_2, y_2) < c(x_1, y_2) + c(x_2, y_1)
\end{equation}
for all $x_1 < x_2$ and $y_1 < y_2$. %
An example of such a cost function is $|x - y|^\lambda$, where
$\lambda > 1$; in this case the quantity
$\mathrm{MK}_\lambda(\hat\mu_0, \hat\mu_1) = (\inf_\gamma \hat
I(\gamma))^{1/\lambda}$ turns out to be a metric on the set of
measures on the circle, referred to as the \emph{Monge--Kantorovich
  distance} of order $\lambda$. %
The value $\lambda = 1$ can still be treated in the same framework as
the limiting case $\lambda \to 1$; it is sometimes called the
Kantorovich--Rubinshtein metric or, in image processing literature,
the Earth Mover's distance~\cite{Rubner.Y:2000}. %

The Monge condition~\eqref{eq:3} implies that whenever under a
transport plan the mutual order of any two elements of mass is
reversed, the transport cost can be strictly reduced by exchanging
their destinations. %
It follows that a locally minimal transport plan moves elements of
mass \emph{monotonically}, preserving their spatial order. %

\begin{figure}
  \centering
  \begin{tikzpicture}
    \begin{scope}[xscale = 1.6, yscale = 1.6, >=stealth]
      \draw[->] (0, -1) -- (0, 3) node [left] {$v$}; %
      \draw[->] (-.5, 0) -- (3.5, 0) node [above] {$u$} ; %
      \draw (0, 0) node [above left] {$O$}; %
      \draw[very thick] (-.5, -.5) cos (0, 0) sin (.5,
      .5) cos (1, 1) sin (1.5, 1.5) cos (2, 2) sin (2.5, 2.5) cos (3,
      3); %
      \draw (-.6, -.3) -- (0, 0) -- (.4, .4) -- (.4, .7) -- (1, 1) --
      (1.4, 1.4) -- (1.4, 1.7) -- (2, 2) -- (2.4, 2.4) -- (2.4, 2.7)
      -- (2.8, 2.9) node [above left] {$F_1$}; %
      \begin{scope}[yshift = -.8cm]
        \draw[very thick] (-.2, -.1) -- (0, 0) -- (.4, .4) -- (.4, .7)
        -- (1, 1) -- (1.4, 1.4) -- (1.4, 1.7) -- (2, 2) -- (2.4, 2.4)
        -- (2.4, 2.7) -- (3, 3) node [below right] {$F_1^\theta$} --
        (3.2, 3.2); %
        \draw (0, 0) node [right] {$\;-\theta$}; %
      \end{scope}
      \draw[densely dotted] (1.2, 0) -- (1.2, 1.33) -- (2.13, 1.33)
      -- (2.13, 0) node [below] {$\quad\mathstrut
        (F_1^{\smash\theta})^{\smash-1}(v)$}; %
      \draw (1.2, 0) node [below] {$\mathstrut
        F_0^{\smash-1}(v)\quad$}; %
      \draw[densely dotted] (1.2, 1.33) -- (0, 1.33) node [left]
      {$v$}; %
    \end{scope}
  \end{tikzpicture}
  \caption{Construction of the locally optimal transport
    plan~$\gamma_\theta$.}
  \label{fig:graphs}
\end{figure}
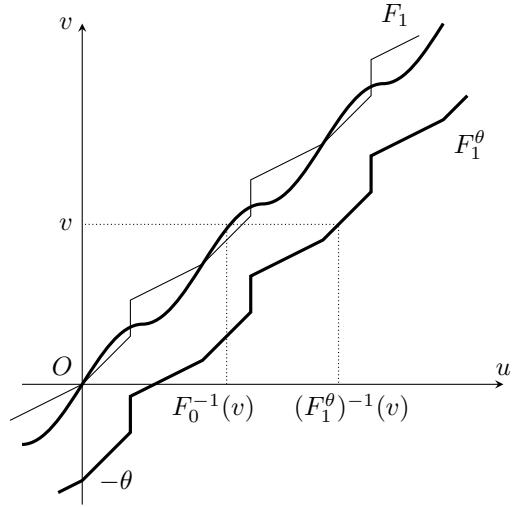

The whole set of locally optimal transport plans for a given pair of
marginals $\mu_0$,~$\mu_1$ can be conveniently described using a
construction represented in fig.~\ref{fig:graphs}. %
Let $F_0$,~$F_1$ be cumulative distribution functions of the measures
$\mu_0$,~$\mu_1$ normalized so that $F_0(0) = F_1(0) = 0$. %
We shall regard graphs of $F_0$,~$F_1$ as continuous curves including,
where necessary, the vertical segments corresponding to jumps of these
functions (which are caused by atoms of $\mu_0$,~$\mu_1$). %
Each of these curves specifies a correspondence, $F_0^{-1}$
or~$F_1^{-1}$, between points of the vertical axis $Ov$, representing
elements of mass, and points of the horizontal axis $Ou$, representing
spatial locations, and maps the Lebesgue measure on~$Ov$ into $\mu_0$
or~$\mu_1$ on~$Ou$. %
This correspondence is monotone and defined everywhere except on an
(at most countable) set of $v$ values that correspond to vacua of the
measure in the $Ou$ axis. %

Define now $F_1^\theta(u) = F_1(u) - \theta$. %
Then $(F_1^\theta)^{-1}$ represents a \emph{shift} of the $Ov$ axis
by~$\theta$ followed by an application of the
correspondence~$F_1^{-1}$, and still induces on the $Ou$ axis the same
measure~$\mu_1$ as $F_1$. %
A transport plan $\gamma_\theta$ that takes an element of mass
represented by~$v$ from $F_0^{-1}(v)$ to~$(F_1^\theta)^{-1}(v)$ is, by
construction, a monotone coupling of $\mu_0$ and~$\mu_1$, and thus a
locally optimal transport plan. %
Moreover, it is shown in \S\ref{sec:conjugate} that all locally
optimal transport plans can be obtained using this construction as the
parameter~$\theta$ runs over~$\rset$. %

Finally define the \emph{average cost} $C_{[F_0, F_1]}(\theta)$ of the plan
$\gamma_\theta$ per unit period:
\begin{displaymath}
  C_{[F_0, F_1]}(\theta) = \int_0^1 c(F_0^{-1}(v), (F_1^\theta)^{-1}(v))\, \diff v.
\end{displaymath}
It is shown in \S\ref{sec:average-cost} that the Monge condition
implies convexity of~$C_{[F_0, F_1]}(\theta)$ and that the global
minimum of this function in~$\theta$ coincides with the minimum value
of the transport cost on the unit circle~\eqref{eq:1}. %

When the marginals $\mu_0$,~$\mu_1$ are purely atomic with finite
numbers $n_0$ and~$n_1$ of atoms in each period, the function $C$
becomes piecewise affine. %
In \S\ref{sec:algor-transp-optim} we present an algorithm to
approximate its minimum value to accuracy~$\epsilon$, using a binary
search that takes $O((n_0 + n_1)\log(1/\epsilon))$ operations in the
real number computing model. %
When masses of all atoms are rational numbers with the least common
denominator~$M$, this search returns an exact solution provided that
$\epsilon < 1/M$. %
This gives an $O((n_0 + n_1)\log M)$ exact transport optimization
algorithm on the circle. %

\section{Preliminaries}
\label{sec:preliminaries}

Let $\tset = \rset/\zset$ be the unit circle, i.e., the segment
$[0,1]$ with identified endpoints. %
By $\pi\colon \rset \to \tset$ denote the projection that takes points
of the universal cover~$\rset$ to points of~$\tset$. %

\subsection{The cost function}
\label{sec:cost-function}

A \emph{cost function} is a real-valued function $c(\cdot, \cdot)$
defined on the universal cover~$\rset$ of the circle~$\tset$. %
We assume that it satisfies the \emph{Monge condition}: for any $x_1 <
x_2$ and~$y_1 < y_2$,
\begin{equation}
  \label{eq:4}
  c(x_1, y_1) + c(x_2, y_2) - c(x_1, y_2) - c(x_2, y_1) < 0.
\end{equation}
Additionally $c$ is assumed to be lower semicontinuous, to be
invariant with respect to integer shifts, i.e.,
\begin{equation}
  \label{eq:5}
  c(x + 1, y + 1) = c(x, y)
\end{equation}
for all $x$,~$y$, and to grow uniformly as $|x - y| \to \infty$: for
any~$P$ there exists a finite~$R(P) \ge 0$ such that
\begin{equation}
  \label{eq:6}
   c(x, y) \ge P\quad \text{whenever $|x - y| \ge R(P)$.}
\end{equation}
Note that the latter condition implies that the lower semicontinuous
function~$c$ is bounded from below (and guarantees that the minima in
a number of formulas below are attained).

Note that the Monge condition~\eqref{eq:4} holds for any twice
continuously differentiable function $c$ such that $\partial^2 c(x,
y)/\partial x\,\partial y < 0$. %
If the cost function depends only on $x - y$, this reduces to a
convexity condition: $-\partial^2 c(x - y)/\partial x\,\partial y =
c''(x - y) > 0$. %
In particular, all the above conditions are satisfied for the function
$c(x, y) = |x -y|^\lambda$, which appears in the definition of the
Monge--Kantorovich distance with $\lambda > 1$, and, more generally,
for any function of the form $c(x - y) + f(x) + g(y)$ with strictly
convex~$c$ and periodic $f$ and~$g$. %

For a cost function $c$ satisfying all the above conditions, the cost
of transporting a unit mass from $\hat x$ to~$\hat y$ on the circle is
defined as $\hat c(\hat x, \hat y) = \inf c(x, y)$, where $\hat
x$,~$\hat y$ are points of~$\tset$ and $\inf$ is taken over all
$x$,~$y$ in~$\rset$ such that $\pi x = \hat x$ and~$\pi y = \hat y$. %
Using the integer shift invariance, this definition can be lifted to
the universal cover as $\hat c(x, y) = \inf_{k\in\zset} c(x, y +
k)$. %

Condition~\eqref{eq:4} is all that is needed in
\S\ref{sec:conjugate}, which is concerned with locally optimal
transport plans on~$\rset$. %
Conditions~\eqref{eq:5}, \eqref{eq:6} come into play in
\S\ref{sec:average-cost}, which deals with transport
optimization on the circle.

\subsection{Distribution functions}
\label{sec:distr-funct}

For a given locally finite measure $\mu$ on~$\rset$ define its
\emph{distribution function} $F_\mu$ by
\begin{equation}
  \label{eq:7}
  F_\mu(0) = 0,\quad
  F_\mu(x) = \mu((0, x])\ \text{for $x > 0$},\quad
  F_\mu(x) = -\mu((x, 0])\ \text{for $x < 0$}.
\end{equation}
Then $\mu((x_1, x_2]) = F_\mu(x_2) - F_\mu(x_1)$ whenever~$x_1 < x_2$,
and this identity also holds for any function that differs
from~$F_\mu$ by an additive constant (the normalization $F_\mu(0) = 0$
is arbitrary). %
When $\mu$ is periodic with unit mass in each period, it follows that for all $x$ in~$\rset$
\begin{equation}
  \label{eq:8}
  F_\mu(x + 1) = F_\mu(x) + 1.
\end{equation}

The \emph{inverse} of a distribution function~$F_\mu$ is defined by
\begin{equation}
  \label{eq:9}
  F_\mu^{-1}(y) = \inf\{x\colon y < F_\mu(x)\}
  = \sup \{x\colon y \ge F_\mu(x)\}.
\end{equation}
Definitions~\eqref{eq:7} and~\eqref{eq:9} mean that $F_\mu$,
$F_\mu^{-1}$ are right-continuous. %
Discontinuities of $F_\mu$ correspond to atoms of~$\mu$ and
discontinuities of its inverse, to ``vacua'' of~$\mu$, i.e., to
intervals of zero $\mu$ measure. %

For a distribution function~$F_\mu$ define its \emph{complete graph}
to be the continuous curve formed by the union of the graph of~$F_\mu$
and the vertical segments corresponding to jumps of~$F_\mu$. %
Accordingly, by a slight abuse of notation let $F_\mu(\{x\})$ denote
the set $[F_\mu(x - 0), F_\mu(x)]$ (warning: $F_\mu(\{x\}) \supseteq
\{F_\mu(x)\}$) and let $F_\mu(A) = \bigcup_{x\in A} F_\mu(\{x\})$ for
any set~$A$. 

\subsection{Local properties of transport plans}
\label{sec:local-transp}

Let $\hat\mu_0$, $\hat\mu_1$ be two finite positive measures of unit
total mass on~$\tset$ and $\mu_0$, $\mu_1$ their liftings to the
universal cover~$\rset$, i.e., periodic measures such that $\mu_i(A) =
\hat\mu_i(\pi A)$, $i = 0, 1$, for any Borel set~$A$ that fits inside
one period. %
Periodicity of measures here means that $\mu(A + n) = \mu(A)$ for any
integer~$n$ and any Borel~$A$, where $A + n = \{x + n\colon x\in
A\}$. %

\begin{definition}
  \label{def:finite}
  A~\emph{(locally finite)%
    \footnote{In what follows the words `locally finite' defining a
      transport plan will often be dropped.} %
    transport plan} with marginals $\mu_0$ and~$\mu_1$ is a locally
  finite measure~$\gamma$ on $\rset\times\rset$ such that
  \begin{romannum}
  \item\label{item:1} for any $x$ in~$\rset$ the supports of measures
    $\gamma((-\infty, x]\times\cdot)$ and $\gamma(\cdot\times(-\infty,
    x])$ are bounded from above and the supports of measures
    $\gamma((x, \infty)\times\cdot)$, $\gamma(\cdot\times(x, \infty))$
    are bounded from below;
  \item\label{item:2} $\gamma(A\times\rset) = \mu_0(A)$ and
    $\gamma(\rset\times B) = \mu_1(B)$ for any Borel sets $A$, $B$.
  \end{romannum}
\end{definition}

The quantity $\gamma(A\times B)$ is the amount of mass transferred
from $A$ to~$B$ under the transport plan~$\gamma$. %
Condition~(\ref{item:1}) implies that the mass supported on any
bounded interval gets redistributed over a bounded set (indeed, a
bounded interval is the intersection of two half-lines), but is
somewhat stronger. %

\begin{definition}
  \label{def:optimal}
  A \emph{local modification} of the locally finite transport
  plan~$\gamma$ is a transport plan~$\gamma'$ such that $\gamma$
  and~$\gamma'$ have the same marginals and $\gamma' - \gamma$ is a
  compactly supported finite signed measure. %
  A local modification is called \emph{cost-reducing} if
  \begin{displaymath}
    \iint c(x, y)\, (\gamma'(\diff x\times\diff y) - \gamma(\diff
    x\times\diff y)) < 0.
  \end{displaymath}
  A locally finite transport plan $\gamma$ is said to be \emph{locally
    optimal} with respect to the cost function~$c$ or
  \emph{$c$-locally optimal} if it has no cost-reducing local
  modifications. %
\end{definition}

\section{Conjugate transport plans and shifts}
\label{sec:conjugate}

Let $U_0$,~$U_1$ be two copies of~$\rset$ equipped with positive
periodic measures $\mu_0$,~$\mu_1$ whose distribution functions
$F_0$,~$F_1$ satisfy~\eqref{eq:8}, so that all intervals of unit
length have unit mass. %
Let furthermore $V_0$,~$V_1$ be two other copies of~$\rset$ equipped
with the uniform (Lebesgue) measure. %

\subsection{Normal plans and conjugation}
\label{sec:norm-plans-conj}

We introduce the following terminology:
\begin{definition}
  \label{def:normal}
  A locally finite transport plan~$\nu$ on~$V_0\times V_1$ with uniform
  marginals is called \emph{normal}. %
\end{definition}

\begin{definition}
  \label{def:conjugate}
  For a normal transport plan~$\nu$ its \emph{conjugate transport
    plan} $\nu^{[F_0, F_1]}$ is a transport plan on $U_0\times
  U_1$ such that for any Borel sets~$A$,~$B$
  \begin{equation}
    \label{eq:10}
    \nu^{[F_0,F_1]}(A\times B) = \nu(F_0(A)\times F_1(B)).
  \end{equation}
\end{definition}

\begin{lemma}
  \label{lem:conjugate}
  For a normal transport plan $\nu$ its conjugate $\nu^{[F_0,F_1]}$ is
  a locally finite transport plan on $U_0\times U_1$ with marginals
  $\mu_0$, $\mu_1$. %
\end{lemma}

\begin{proof}
  Since distribution functions $F_0$, $F_1$ and their inverses
  preserve boundedness, condition~(\ref{item:1}) of
  Definition~\ref{def:finite} is fulfilled. %
  Definition~\ref{def:conjugate}, condition~(\ref{item:2}) of
  Definition~\ref{def:finite}, and formula~\eqref{eq:7} together imply
  that
  \begin{displaymath}
    \begin{split}
      \nu^{[F_0,F_1]}((u_1, u_2]\times U_1)
      &= \nu(F_0((u_1, u_2])\times F_1(U_1))
      = \nu([F_0(u_1), F_0(u_2)]\times V_1) \\
      &= F_0(u_2) - F_0(u_1) = \mu_0((u_1, u_2]).
    \end{split}
  \end{displaymath}
  Similarly $\nu^{[F_0,F_1]}(U_0\times (u_1, u_2]) = \mu_1((u_1,
  u_2])$. %
  Thus $\nu^{[F_0, F_1]}$ satisfies condition~(\ref{item:2}) of
  Definition~\ref{def:finite} on intervals and therefore on all Borel
  sets. %
\end{proof}

\begin{lemma}
  \label{lem:reverse}
  For any transport plan $\gamma$ on~$U_0\times U_1$ with marginals
  $\mu_0$ and~$\mu_1$ there exists a normal transport plan~$\nu$ such
  that $\gamma$ is conjugate to~$\nu$: $\gamma = \nu^{[F_0,F_1]}$.%
\end{lemma}

\begin{proof}
  For non-atomic measures $\mu_0$ and~$\mu_1$ the required transport
  plan is given by the formula $\nu(A\times B) =
  \gamma(F_0^{-1}(A)\times F_1^{-1}(B))$, which is dual
  to~\eqref{eq:10}. %
  However if, e.g., $\mu_0$ has an atom, then the function $F_0^{-1}$
  is constant over a certain interval and maps any subset~$A$ of this
  interval into one point of \emph{fixed} positive measure in~$U_0$, so
  information on the true Lebesgue measure of~$A$ is lost. %
  In this case extra care has to be taken. %
  
  Recall that a locally finite measure has at most a countable set of
  atoms. %
  Let atoms of~$\mu_0$ be located in $(0, 1]$ at points $u_1$, $u_2$,
  \dots\ with masses $m_1$, $m_2$, \dots\,. %
  Since $\gamma(\{u_i\}\times U_1) = \mu_0(\{u_i\}) = m_i > 0$, there
  exists a conditional probability measure $\rho(\cdot\mid u_i) =
  \gamma(\{u_i\}\times \cdot)/m_i$. %
  For a set $A\subset (0, 1]$ define a ``residue'' transport plan
  \begin{displaymath}
    \bar\gamma(A\times B) = \gamma(A\times B)
    - {\textstyle\sum_i}\, m_i\,\delta_{u_i}(A)\, \rho(B\mid u_i),
  \end{displaymath}
  where $\delta_u$ is the Dirac unit mass measure on~$U_0$ concentrated
  at~$u$, and extend $\bar\gamma$ to general~$A$ using periodicity. %
  We thus remove from~$\bar\gamma$ the part of~$\gamma$ whose
  projection to the first factor is atomic. %
  Define a transport plan $\kappa$ on $V_0\times U_1$ by
  \begin{displaymath}
    \kappa(C\times B)
    = {\textstyle\sum_i}\, \lambda(C\cap F_0(\{u_i\}))\, \rho(B\mid u_i)
    + \bar\gamma(F_0^{-1}(C)\times B),
  \end{displaymath}
  where $C$ is a Borel set in~$V_0$ and $\lambda(\cdot)$ denotes the
  Lebesgue measure in~$V_0$. %
  Clearly $\kappa(F_0(A)\times B) = \gamma(A\times B)$. %
  Repeating this construction for the second factor, with $\kappa$ in
  place of~$\gamma$, we get a normal transport plan~$\nu$ such that
  $\gamma(A\times B) = \nu(F_0(A)\times F_1(B))$.
\end{proof}

Since we are ultimately interested in transport optimization with
marginals $\mu_0$,~$\mu_1$ rather than with uniform marginals, two
normal transport plans $\nu_1$, $\nu_2$ will be called
\emph{equivalent} if they have the same conjugate. %
Two different normal transport plans can only be equivalent if one or
both measures $\mu_0$ or~$\mu_1$ have atoms, causing loss of
information on the structure of $\nu$ in segments corresponding to
these atoms. %
The proof of Lemma~\ref{lem:reverse} gives a specific representative
of this equivalence class of normal plans. %

\subsection{Locally optimal normal transport plans are shifts}
\label{sec:locally-optim-plans}

Fix a cost function $c\colon U_0\times U_1\to \rset$ that satisfies the
Monge condition~\eqref{eq:4} and define
\begin{equation}
  \label{eq:11}
  c_{[F_0, F_1]}(v_0, v_1) = c\bigl(F_0^{-1}(v_0), F_1^{-1}(v_1)\bigr).
\end{equation}
For non-atomic measures $\mu_0$,~$\mu_1$, it satisfies the Monge
condition
\begin{displaymath}
  c_{[F_0, F_1]}(v', w') + c_{[F_0, F_1]}(v'', w'')
  - c_{[F_0, F_1]}(v', w'') - c_{[F_0, F_1]}(v'', w') < 0
\end{displaymath}
whenever $v' < v''$ and $w' < w''$; this inequality can only turn into
equality if either $v', v''$ or $w', w''$ correspond to an atom of the
respective marginal ($\mu_0$ or $\mu_1$) of~$\nu^{[F_0, F_1]}$, i.e.,
if $c_{[F_0, F_1]}$ is constant in either first or second argument. %
In spite of this slight violation of definition of
\S\ref{sec:cost-function}, we will still call $c_{[F_0, F_1]}$ a
cost function. %

Here and below, variables $u$, $u'$, $u_0$, $u_1$, \ldots\ are assumed
to take values in~$U_0$ or~$U_1$ and variables $v$, $v'$, $v_0$,
$v_1$, \dots, $w$, $w'$, \dots, in~$V_0$ or~$V_1$. %

\begin{lemma}
  \label{lem:optimal}
  A transport plan $\gamma$ on $U_0\times U_1$ with marginals $\mu_0$,
  $\mu_1$ is $c$-locally optimal if and only if it is conjugate to a
  $c_{[F_0, F_1]}$-locally optimal normal transport plan $\nu$. %
  In particular, all normal transport plans with the same locally
  optimal conjugate are locally optimal.
\end{lemma}

\begin{proof}
  Note that $\nu' - \nu$ is compactly supported if and only if the
  difference of the respective conjugates $\gamma' - \gamma$ is
  compactly supported. %
  The rest of the proof follows from the identity
  \begin{displaymath}
    \begin{split}
      &\iint c_{[F_0, F_1]}(v_1, v_2)\, \bigl(\nu'(\diff v_1\times\diff v_2) -
      \nu(\diff v_1\times\diff v_2)\bigr) \\
      &\quad = \iint c(u_1, u_2)\, \bigl(\gamma'(\diff u_1\times\diff
      u_2) - \gamma(\diff u_1\times\diff u_2)\bigr)
    \end{split}
  \end{displaymath}
  established by the change of variables $v_1 = F_0(u_1)$, $v_2 =
  F_1(u_2)$ (here jumps of the distribution functions are harmless
  because $c_{[F_0, F_1]}$ is constant over respective ranges of its
  variables).
\end{proof}

Transport optimization with marginals $\mu_0$,~$\mu_1$ is thus reduced
to a \emph{conjugate problem} involving uniform marginals and the
cost~$c_{[F_0, F_1]}$. %
It turns out that any $c_{[F_0, F_1]}$-optimal normal transport plan
must be supported on a graph of a monotone function, and due to
uniformity of marginals this function can only be a shift by a
suitable real increment $\theta$. %
More precisely, the following holds:
\begin{theorem}
  \label{thm:shifts}
  Let $\mu_0$, $\mu_1$ be two periodic positive measures defined respectively on~$U_0$,~$U_1$ with
  unit mass in each period and let $F_i\colon U_i\to V_i$, $i = 0, 1$, be their
  distribution functions. %
  Then any $c_{[F_0, F_1]}$-locally optimal normal transport plan
  on~$V_0\times V_1$ is equivalent to a normal transport plan~$\nu_\theta$
  with $\mathop{\mathrm{supp}}\nu_\theta = \{(v, w)\colon w = v +
  \theta\}$, and conversely $\nu_\theta$ is $c_{[F_0, F_1]}$-locally
  optimal for any real~$\theta$. %
  All $c$-locally optimal transport plans on~$U_0\times U_1$ with
  marginals $\mu_0$,~$\mu_1$ are of the form $\gamma_\theta =
  (\nu_\theta)^{[F_0, F_1]}$. %
\end{theorem}

The proof, divided into a series of lemmas, is based on the classical
argument: a nonoptimal transport plan can be modified by ``swapping''
pieces of mass to render its support monotone while decreasing its
cost. %
This argument, carried out for plans with uniform marginals
on~$V_0\times V_1$, is combined with the observation that a monotonicaly
supported plan with uniform marginals can only be a shift. %
Then Lemma~\ref{lem:optimal} is used to extend this result to
transport plans on~$U_0\times U_1$.

Throughout the proof fix a normal transport plan~$\nu$ and define on $V_0\times V_1$ the functions
\begin{equation}
  \label{eq:12}
  r_\nu(v, w) = \nu((-\infty, v]\times (w, \infty)),\quad
  l_\nu(v, w) = \nu((v, \infty)\times (-\infty, w]).
\end{equation}
To explain the notation $r_\nu$,~$l_nu$ observe that, e.g., $r_\nu(v, w)$
is the amount of mass that is located initially to the left of~$v$ and
goes to the \emph{right} of~$w$. 

\begin{lemma}
  \label{lem:fg}
  The function $r_\nu$ (resp.\ $l_\nu$) is continuous and
  monotonically increasing in its first (second) argument and is
  continuous and monotonically decreasing in its second (first)
  argument, while the other argument is kept fixed. %
\end{lemma}

\begin{proof}
  Monotonicity is obvious from~\eqref{eq:12}. %
  To prove continuity observe that the second marginal of~$\nu$
  is uniform, which together with positivity of all involved
  measures implies that in the decomposition
  \begin{displaymath}
    \nu(V_0\times\cdot\,)
    = \nu((-\infty, v]\times \cdot\,) + \nu((v, \infty)\times\cdot\,),
  \end{displaymath}
  both measures in the right-hand side cannot have atoms. %
  This implies continuity of $r_\nu$,~$l_\nu$ with respect to the
  second argument. %
  A~similar proof holds for the first argument.
\end{proof}

\begin{lemma}
  \label{lem:m}
  For any $v$ there exist $w_\nu(v)$ and $m_\nu(v) \ge 0$ such that
  \begin{equation}
    \label{eq:13}
    r_\nu(v, w_\nu(v)) = l_\nu(v, w_\nu(v)) = m_\nu(v).
  \end{equation}
  The correspondence $v\mapsto w_\nu(v)$ is monotone: $w_\nu(v_1) \le w_\nu(v_2)\
  \text{for}\ v_1 < v_2$.
\end{lemma}

\begin{proof}
  Clearly $r_\nu(v, -\infty) = \infty$, $r_\nu(v, \infty) = 0$,
  $l_\nu(v, -\infty) = 0$, $l_\nu(v, \infty) = \infty$. %
  The continuity of the functions $r_\nu(v, \cdot)$,~$l_\nu(v, \cdot)$
  in the second argument for a fixed~$v$ implies that their graphs
  intersect at some point~$(w_\nu(v), m_\nu(v))$, which
  satisfies~\eqref{eq:13}. %
  Should the equality $r_\nu(v, w) = l_\nu(v, w)$ hold on a segment
  $[w', w'']$, we set $w_\nu(v)$ to its left endpoint~$w'$; this
  situation, however, will be ruled out by the corollary to
  Lemma~\ref{lem:monotone} below. %
  Monotonicity of $w_\nu(v)$ follows from monotonicity of
  $r_\nu(\cdot, w)$, $l_\nu(\cdot, w)$ in the first argument for a
  fixed~$w$: indeed, for $v_2 > v_1$ the equality $r_\nu(v_2, w) =
  l_\nu(v_2, w)$ is impossible for $w < w_\nu(v_1)$ because for such
  $w$ we have $r_\nu(v_2, w) > r_\nu(v_1, w_\nu(v)) = l_\nu(v_1,
  w_\nu(v)) > l_\nu(v_2, w)$. %
\end{proof}

Equalities~\eqref{eq:13} mean that the same amount of mass~$m_\nu(v)$
goes under the plan~$\nu$ from the left of~$v$ to the right
of~$w_\nu(v)$ and from the right of~$v$ to the left of~$w_\nu(v)$. %
We are now in position to use the Monge condition and show that this
amount can be reduced to zero by modifying the transport plan locally
without a cost increase. %

\begin{lemma}
  \label{lem:onex}
  For any~$v$ there exists a local modification~$\nu_v$ of~$\nu$ such
  that $w_{\nu_v}(v) = w_\nu(v)$ (with $w_\nu$ defined as in
  Lemma~\ref{lem:m}), $m_{\nu_v}(v) = 0$, and $\nu_v$ is either
  cost-reducing in the sense of Definition~\ref{def:optimal} or is
  equivalent to~$\nu$.
\end{lemma}

\begin{proof}
  Let $w = w_\nu(v)$ and $m = m_\nu(v)$. %
  If $m = 0$, there is nothing to prove.  %
  Suppose that $m > 0$ and define
  \begin{displaymath}
    \begin{aligned}
      w^- &= \sup\{w'\colon l_\nu(v, w') = 0\},\quad&
      w^+ &= \inf\{w'\colon r_\nu(v, w') = 0\},\\
      v^- &= \sup\{v'\colon r_\nu(v', w) = 0\},\quad&
      v^+ &= \inf\{v'\colon l_\nu(v', w) = 0\}.
    \end{aligned}
  \end{displaymath}
  By local finiteness of the transport plan~$\nu$ all these quantities
  are finite. %
  Since $m > 0$, continuity of $r_\nu$,~$l_\nu$ implies that the
  inequalities $w^- < w < w^+$ and $v^- < v < v^+$ are strict. %
  Consider the measures
  \begin{displaymath}
    \begin{aligned}
      \rho^-(\cdot) &= \nu(\,\cdot\times(w, w^+))\ \text{on~$(v^-, v)$},\quad&
      \rho^+(\cdot) &= \nu(\,\cdot\times(w^-, w))\ \text{on~$(v, v^+)$},\\
      \sigma^-(\cdot) &= \nu((v, v^+)\times\cdot\,)\ \text{on~$(w^-, w)$},\quad&
      \sigma^+(\cdot) &= \nu((v^-, v)\times\cdot\,)\ \text{on~$(w, w^+)$}.
    \end{aligned}
  \end{displaymath}
  Equalities~\eqref{eq:13} mean that all these measures have the same
  positive total mass$~m$. %
  Note that the Lebesgue measures of intervals $(v^-, v)$, $(v, v^+)$,
  $(w^-, w)$, and $(w, w^+)$ may be greater than~$m$, because some
  mass in these intervals may come from or go to elsewhere. %

  The functions $r_w(\cdot) = r_\nu(\cdot, w)$, $l_v(\cdot) = l_\nu(v,
  \cdot)$ are monotonically increasing and $r_v(\cdot) = r_\nu(v,
  \cdot)$, $l_w(\cdot) = l_\nu(\cdot, w)$ are monotonically
  decreasing, with their inverses $r_w^{-1}$, $l_v^{-1}$, $r_v^{-1}$,
  $l_w^{-1}$ defined everywhere except on an at most countable set of
  points. %
  These functions may be regarded as a kind of distribution functions
  for the measures $\rho^-$, $\sigma^-$, $\sigma^+$, $\rho^+$
  respectively, mapping them to the Lebesgue measure on~$(0, m)$.

  Under the plan~$\nu$, mass~$m$ is sent from $(v^-, v)$ to $(w, w^+)$
  and from $(v, v^+)$ to $(w^-, w)$. %
  We now construct a local modification~$\nu_v$ of the transport
  plan~$\nu$ that moves mass~$m$ from the interval~$(v^-, v)$
  to~$(w^-, w)$ and from~$(v, v^+)$ to~$(w, w^+)$, and show that it is
  cost-reducing unless measures $\mu_0$,~$\mu_1$ have atoms
  corresponding to the intervals under consideration. %

  Observe first that the normal plan~$\nu$ induces two transport plans
  $\tau_r$,~$\tau_l$ that map measures $\rho^-$ to~$\sigma^+$ and
  $\rho^+$ to~$\sigma^-$ correspondingly:
  \begin{displaymath}
    \begin{gathered}
      \tau_r(A\times B)
      = \nu(A\cap(-\infty, v)\times B\cap(w, +\infty))
      = \nu(A\cap(v^-, v)\times B\cap(w, w^+)), \\
      \tau_l(A\times B)
      = \nu(A\cap(v, +\infty)\times B\cap(-\infty, w))
      = \nu(A\cap(v, v^+)\times B\cap(w^-, w)),
    \end{gathered}
  \end{displaymath}
  where $A\subset (v^-, v^+)$, $B\subset (w^-, w^+)$ are two arbitrary
  Borel sets and the $\cap$ operation takes precedence
  over~$\times$. %
  By an argument similar to the proof of Lemma~\ref{lem:reverse},
  there exist two transport plans $\chi_r$ and~$\chi_l$ mapping the
  Lebesgue measure on~$(0, m)$ respectively to $\sigma^+$,~$\sigma^-$
  and such that
  \begin{displaymath}
    \begin{gathered}
      \tau_r(A\times B) = \chi_r\bigl(r_w(A\cap(v^-, v))\times
      B\cap(w, w^+)\bigr), \\
      \tau_l(A\times B) = \chi_l\bigl(l_w(A\cap(v, v^+))\times
      B\cap(w^-, w)\bigr).
    \end{gathered}
  \end{displaymath}
  Define now two transport plans $\bar\tau_l$, $\bar\tau_r$ that send
  mass elements to the same destinations but from interchanged origins:
  \begin{displaymath}
    \begin{gathered}
      \bar\tau_r(A\times B) = \chi_r\bigl(l_w(A\cap(v,
      v^+))\times B\cap(w, w^+)\bigr), \\
      \bar\tau_l(A\times B) = \chi_l\bigl(r_w(A\cap(v^-, v))\times
      B\cap(w^-, w)\bigr).
    \end{gathered}
  \end{displaymath}
  This enables us to define
  \begin{displaymath}
    \nu_v(A\times B) = \nu(A\times B)
    - \tau_r(A\times B) - \tau_l(A\times B)
    + \bar\tau_r(A\times B) + \bar\tau_l(A\times B).
  \end{displaymath}
  Since $\tau_r(A\times\rset) = \bar\tau_r(A\times\rset) = \rho^-(A)$
  etc., the transport plan~$\nu_v$ has the same uniform marginals
  as~$\nu$, i.e., it is a local modification of~$\nu$. %
  Observe furthermore that by the construction of~$\nu_v$ no mass is
  moved under this plan from the left-hand side of~$v$ to the
  right-hand side of~$w$ and inversely, i.e., that $m_{\nu_v}(v) =
  0$. %

  It remains to show that $\nu_v$ is either a cost-reducing
  modification of~$\nu$ or equivalent to it. %
  By the disintegration lemma (see, e.g., \cite{Ambrosio.L:2005}) we
  can write $\chi_r(\diff\alpha\times\diff w') = \diff\alpha\; \diff
  G_r(w'\mid\alpha)$ and $\chi_l(\diff\alpha\times\diff w') =
  \diff\alpha\; \diff G_l(w'\mid\alpha)$, where
  $G_r(\,\cdot\mid\alpha)$ (resp.\ $G_l(\,\cdot\mid\alpha)$) are
  distribution functions of probability measures defined on $[w, w^+]$
  (resp.\ $[w^-, w]$) for almost all $0 < \alpha < m$. %
  Denote their respective inverses by $G_r^{-1}(\cdot\mid\alpha)$,
  $G_l^{-1}(\cdot\mid\alpha)$ and observe that $w^-\le
  G_l^{-1}(\beta'\mid\alpha) \le w\le G_r^{-1}(\beta''\mid\alpha) \le
  w^+$ for any $\beta'$,~$\beta''$. %
  Thus
  \begin{displaymath}
    \begin{split}
      \iint c(v', w')\; \tau_r(\diff v'\times\diff w') & = \iint
      c(r_w^{-1}(\alpha), w')\;
      \chi_r(\diff\alpha\times\diff w') \\
      & = \int_0^m\diff\alpha \int c(r_w^{-1}(\alpha),
      w')\, \diff G_r(w'\mid\alpha) \\
      & = \int_0^m\diff\alpha \int_0^1\diff\beta\ c(r_w^{-1}(\alpha),
      G_r^{-1}(\beta\mid\alpha)),
    \end{split}
  \end{displaymath}
  where we write $c$ instead of $c_{[F_0, F_1]}$ to lighten notation,
  and similarly
  \begin{displaymath}
  \begin{gathered}
    \iint c(v', w')\; \tau_l(\diff v'\times\diff w')
    = \int_0^m \diff\alpha \int_0^1 \diff\beta\
    c(l_w^{-1}(\alpha), G_l^{-1}(\beta\mid\alpha)),\\
    \iint c(v', w')\; \bar\tau_r(\diff v'\times\diff w')
    = \int_0^m \diff\alpha \int_0^1 \diff\beta\
    c(l_w^{-1}(\alpha), G_r^{-1}(\beta\mid\alpha)),\\
    \iint c(v', w')\; \bar\tau_l(\diff v'\times\diff w')
    = \int_0^m \diff\alpha \int_0^1 \diff\beta\
    c(r_w^{-1}(\alpha), G_l^{-1}(\beta\mid\alpha)).
  \end{gathered}
\end{displaymath}
  The integral in Definition~\ref{def:optimal} now takes the form
  \begin{displaymath}
    \begin{split}
      \iint c(v', w')\, (\nu_v(\diff v'&\times\diff w') -
      \nu(\diff v'\times\diff w')) \\
      \quad =\iint c(v', w')\, &\bigl(-\tau_r(\diff
      v'\times\diff w') - \tau_l(\diff v'\times\diff w') \\
      &\ + \bar\tau_r(\diff v'\times\diff w') + \bar\tau_l(\diff
      v'\times\diff w')\bigr) \\
      = \int_0^m\!\!\!\diff\alpha \int_0^1\!\!\diff\beta\; &
      \bigl(-c(r_w^{-1}(\alpha), G_r^{-1}(\beta\mid\alpha))
      - c(l_w^{-1}(\alpha), G_l^{-1}(\beta\mid\alpha)) \\
      &\ + c(l_w^{-1}(\alpha), G_r^{-1}(\beta\mid\alpha)) + 
      c(r_w^{-1}(\alpha), G_l^{-1}(\beta\mid\alpha))\bigr).
    \end{split}
  \end{displaymath}
  As $r_w^{-1}(\alpha) \le v \le l_w^{-1}(\alpha)$ and
  $G_l^{-1}(\beta\mid\alpha) \le w \le G_r^{-1}(\beta\mid\alpha)$ for
  all $\alpha$,~$\beta$, the Monge condition~\eqref{eq:4} implies that
  either the value of this integral is negative or 
  the function $c$ (i.e., $c_{[F_0, F_1]}$) is constant in at least
  one of its arguments. %
  In the former case the transport plan~$\nu_v$ is a cost-reducing
  local modification of~$\nu$; in the latter case $\nu_v$ is
  equivalent to~$\nu$. %
\end{proof}

\begin{lemma}
  \label{lem:monotone}
  For any $v' < v''$ there exists a local modification $\nu_{v',v''}$
  of~$\nu$ such that $w_{\nu_{v',v''}}(v) = w_\nu(v)$ for $v' \le v
  \le v''$, $m_{\nu_{v',v''}}(v') = m_{\nu_{v',v''}}(v'') = 0$, and in
  the strip $v' \le v \le v''$ the support of $\nu_{v',v''}$ coincides
  with the complete graph of the monotone function~$w_\nu(\cdot)$.
\end{lemma}

\begin{proof}
  Let $\{v_i\}$ be a dense countable subset of~$[v', v'']$ including
  its endpoints. %
  Set $\nu_0 = \nu$ and define $\nu_i$ recursively to be the local
  modification of~$\nu_{i - 1}$ given by the previous lemma and such
  that $w_{\nu_i}(v_i) = w_\nu(v_i)$ and $m_{\nu_i}(v_i) = 0$. %
  Then all $\nu_i$ are either cost-reducing or equivalent to~$\nu$ and
  $w_{\nu_j}(v_i) = w_\nu(v_i)$, $m_{\nu_j}(v_i) = 0$ for all $j >
  i$. %
  Indeed, denote $w_i = w_{\nu_i}(v_i)$ and observe that if e.g.\ $v_j
  > v_i$, then, as $m_{\nu_i}(v_i) = r_{\nu_i}(v_i, w_i) = 0$, mass
  from $(-\infty, v_i]$ does not appear to the right of~$w_i$ and so
  does not contribute to the balance of mass around~$w_j$. %
  Therefore for any $j$ the possible modification of~$\nu_{j - 1}$ is
  local to the interval $(v_{i'}, v_{i''})$, where $v_{i'} =
  \max\{v_i\colon i < j, v_i < v_j\}$ and $v_{i''} = \min\{v_i\colon i
  < j, v_i > v_j\}$ (with $\max$ and~$\min$ of empty set defined as
  $v'$ and~$v''$). %
  Thus there is a well-defined limit normal transport
  plan~$\nu_\infty$ that is either a cost-reducing local modification
  or equivalent to~$\nu$ and is such that, by continuity of the
  functions $r_{\nu_\infty}$ and~$l_{\nu_\infty}$ in the first
  argument, $m_{\nu_\infty}(v)$ vanishes everywhere on~$[v', v'']$. %
  
  Consider now the function $w_{\nu_\infty}(\cdot)$, which coincides
  with~$w_\nu(\cdot)$ on a dense subset of~$[v', v'']$, so that their
  complete graphs coincide. %
  For any quadrant of the form $(-\infty, v_0)\times(w_0, \infty)$
  such that $w_0 > w_{\nu_\infty}(v_0)$, monotonicity
  of~$r_{\nu_\infty}$ in the second argument implies that
  \begin{displaymath}
    0\le \nu_\infty((-\infty, v_0)\times (w_0, \infty)) =
    r_{\nu_\infty}(v_0, w_0) \le r_{\nu_\infty}(v_0,
    w_{\nu_\infty}(v_0)) = 0,
  \end{displaymath}
  i.e., $\nu_\infty((-\infty, v_0)\times
  (w_0, \infty)) = 0$. %
  Similarly $\nu_\infty((v_0, \infty)\times(-\infty, w_0)) = 0$ for
  any quadrant with $w_0 < w_{\nu_\infty}(v_0)$. %
  The union of all such quadrants is the complement of the complete
  graph of the function $v\mapsto w_\nu(v)$; this implies that
  $\nu_\infty$ is supported thereon.
\end{proof}

\begin{corollary}
  For any normal transport plan $\nu$ there exists a real number
  $\theta_\nu$ such that $w_\nu(v) = v + \theta_\nu$.
\end{corollary}

\begin{proof}
  It is enough to show that $w_\nu(v') - v' = w_\nu(v'') - v''$ for
  all $v'$,~$v''$. %
  Let $v' < v''$ and $\nu_{v', v''}$ be the local modification
  constructed in the previous lemma. %
  Since it has uniform marginals and monotone support, we have
  $w_\nu(v'') - w_\nu(v') = \nu_{v', v''}((v', v'')\times(w_\nu(v'),
  w_\nu(v''))) = v'' - v'$, which completes the proof.
\end{proof}

We call the parameter $\theta_\nu$ the \emph{rotation number} of the
normal transport plan~$\nu$. %

\begin{definition}
  A normal transport plan consisting of a uniform measure supported on
  the line $\{(v, w)\colon w = v + \theta\}$ is called a \emph{shift}
  and denoted by~$\nu_\theta$. %
\end{definition}

\begin{lemma}
  \label{lem:reverserotation}
  For any~$\theta$ the shift $\nu_\theta$ is $c_{[F_0, F_1]}$-locally
  optimal.
\end{lemma}

\begin{proof}
  Let $\bar\nu$ be a local modification of~$\nu_\theta$ such that the
  signed measure~$\nu_\theta - \bar\nu$ is supported in $(v',
  v'')\times(w', w'')$. %
  Let $\bar\nu_{v', v''}$ be a local modification of~$\bar\nu$
  constructed in Lemma~\ref{lem:monotone}; it coincides
  with~$\nu_\theta$ over $v' < v < v''$, and hence everywhere. %
  Since it is either cost-reducing or equivalent to~$\bar\nu$, it
  follows that $\bar\nu$ cannot be cost-reducing with respect
  to~$\nu_\theta$, i.e., that $\nu_\theta$ is a cost minimizer with
  respect to local modifications.
\end{proof}

\begin{lemma}
  \label{lem:rotation}
  Any $c_{[F_0, F_1]}$-locally optimal normal transport $\nu$ with
  rotation number~$\theta = \theta_\nu$ is equivalent to the shift
  $\nu_\theta$.
\end{lemma}

\begin{proof}
  Let $v''_i = -v'_i = i$ for $i = 1, 2, \dots$\, . %
  All local modifications $\nu_i = \nu_{v'_i, v''_i}$ of~$\nu$
  constructed as in Lemma~\ref{lem:monotone} cannot be cost-reducing
  and are therefore equivalent to~$\nu$. %
  On the other hand, this sequence stabilizes to the
  shift~$\nu_\theta$ on any bounded subset of $V_0\times V_1$ as soon
  as this set is covered by $(-i, i)\times(-i, i)$. %
  Therefore $\nu_\theta$ has the same conjugate as all~$\nu_i$ and is
  equivalent to~$\nu$.
\end{proof}

Lemmas~\ref{lem:reverserotation}, \ref{lem:rotation},
and~\ref{lem:optimal} together imply Theorem~\ref{thm:shifts}.

\section{Transport optimization for periodic measures}
\label{sec:average-cost}

Let now $c$ be a cost function that satisfies the Monge
condition~\eqref{eq:4}, the integer shift invariance
condition~\eqref{eq:5}, the growth condition~\eqref{eq:6}, and is
bounded from below. %
Suppose that $\gamma_\theta$ is a locally optimal transport plan
on~$U_0\times U_1$ with marginals $\mu_0$,~$\mu_1$ conjugate to the
shift~$\nu_\theta$. %
Define $c_{[F_0, F_1]}$ as in~\eqref{eq:11} and let $F_1^\theta(u) =
F_1(u) - \theta$ as illustrated in fig.~\ref{fig:graphs}.
\begin{definition}
  \label{def:cost}
  We call the quantity
  \begin{equation}
    \label{eq:14}
    C_{[F_0, F_1]}(\theta) = \int_0^1 c_{[F_0, F_1]}(v', v' + \theta)\,\diff v'
    = \int_0^1 c\bigl(F_0^{-1}(v'), (F_1^\theta)^{-1}(v')\bigr)\, \diff v'
  \end{equation}
  the \emph{average cost} (per period) of the transport
  plan~$\gamma_\theta$. %
\end{definition}
Observe that it is indifferent whether to integrate here from $0$
to~$1$ or from $v$ to~$v + 1$ for any real~$v$. %
Examples of average cost functions $C_{[F_0, F_1]}$ for different
marginals $\mu_0$,~$\mu_1$ and different cost functions $c$ are shown
in fig.~\ref{fig:cost_function}. %

\begin{figure}[htp]
  \centerline{\subfigure[Black bars: $\mu'_0$, light bars: $\mu'_1$]%
    {\includegraphics[width=6cm]{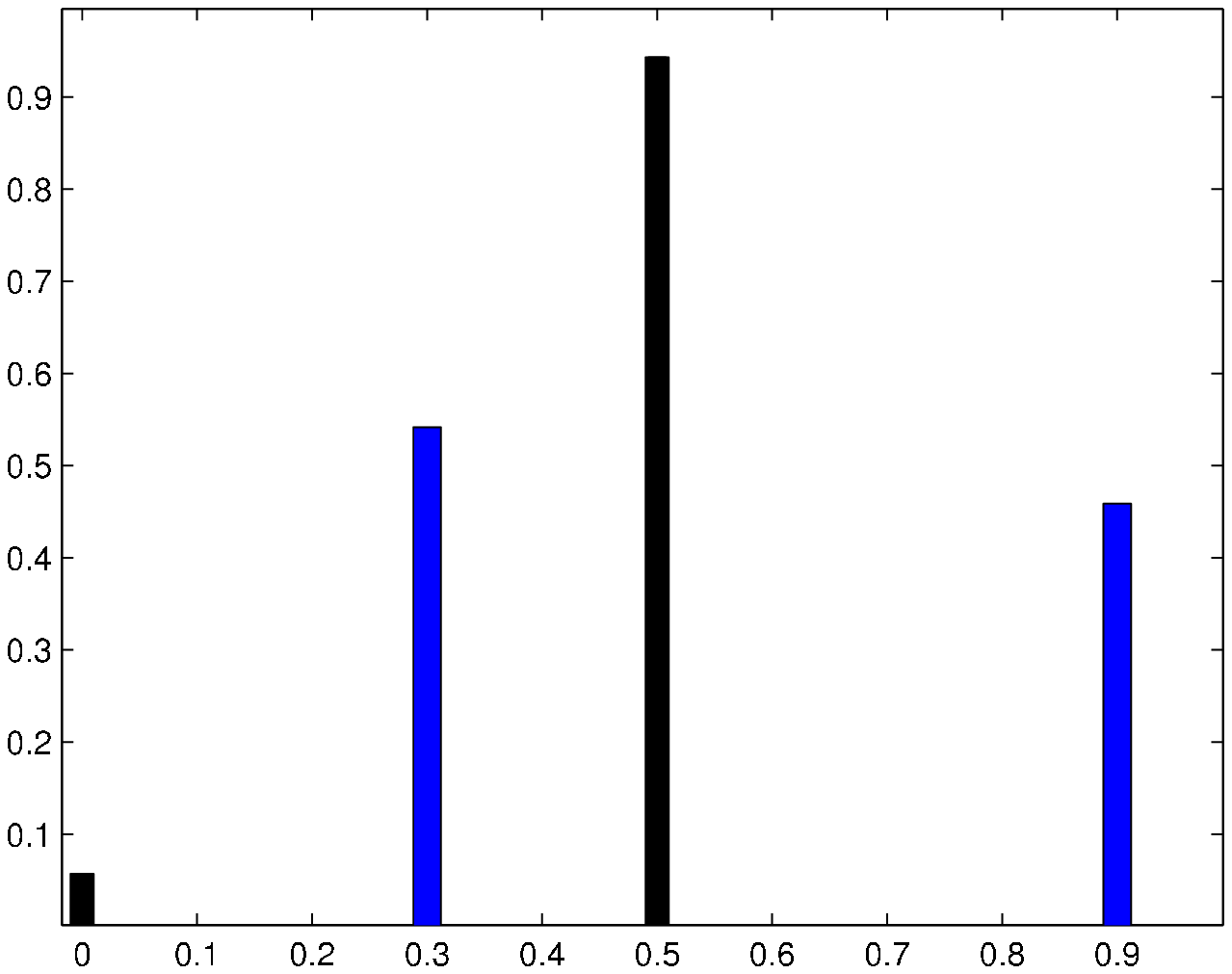}}\hspace{0.5cm} %
    \subfigure[$C_{[F'_o, F'_q]}(\cdot)$ for $c(x, y) = |x - y|$]%
    {\includegraphics[width=6cm]{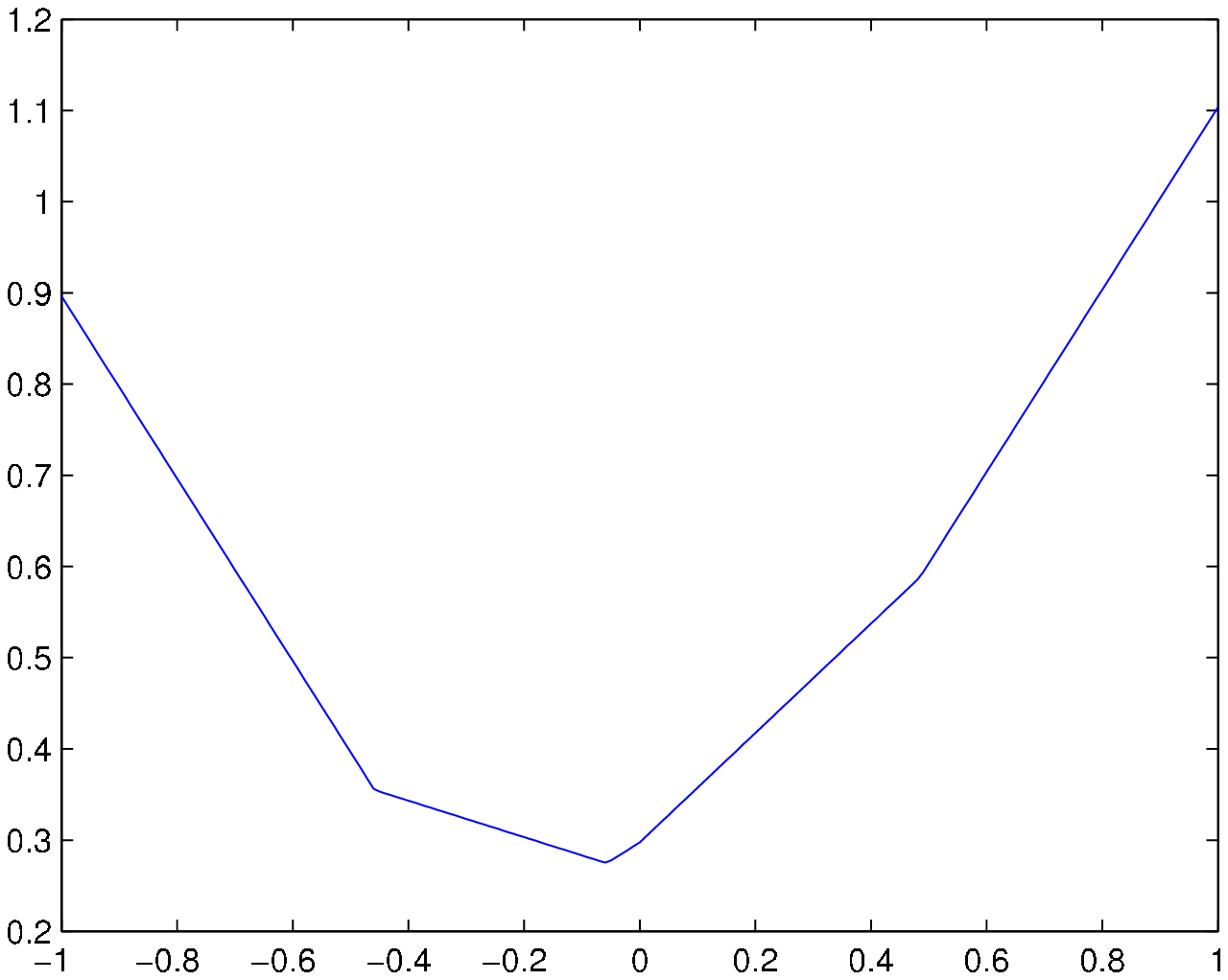}}}

  \centerline{\subfigure[Black bars: $\mu''_0$, light bars: $\mu''_1$]%
    {\includegraphics[width=6cm]{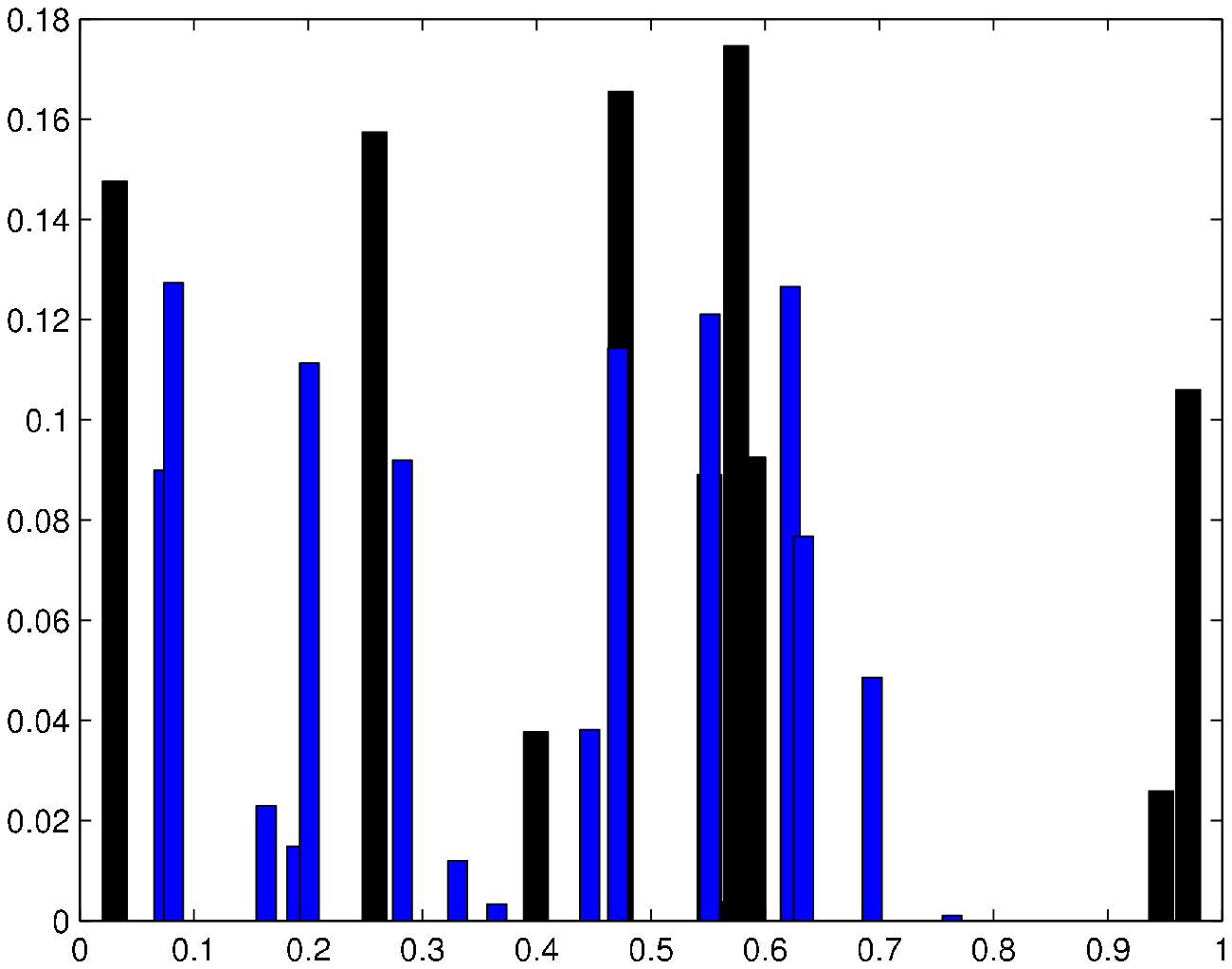}}\hspace{0.5cm} %
    \subfigure[$C_{[F''_0, F''_1]}(\cdot)$ for $c(x,y)= |x-y|^2$]%
    {\includegraphics[width=6cm]{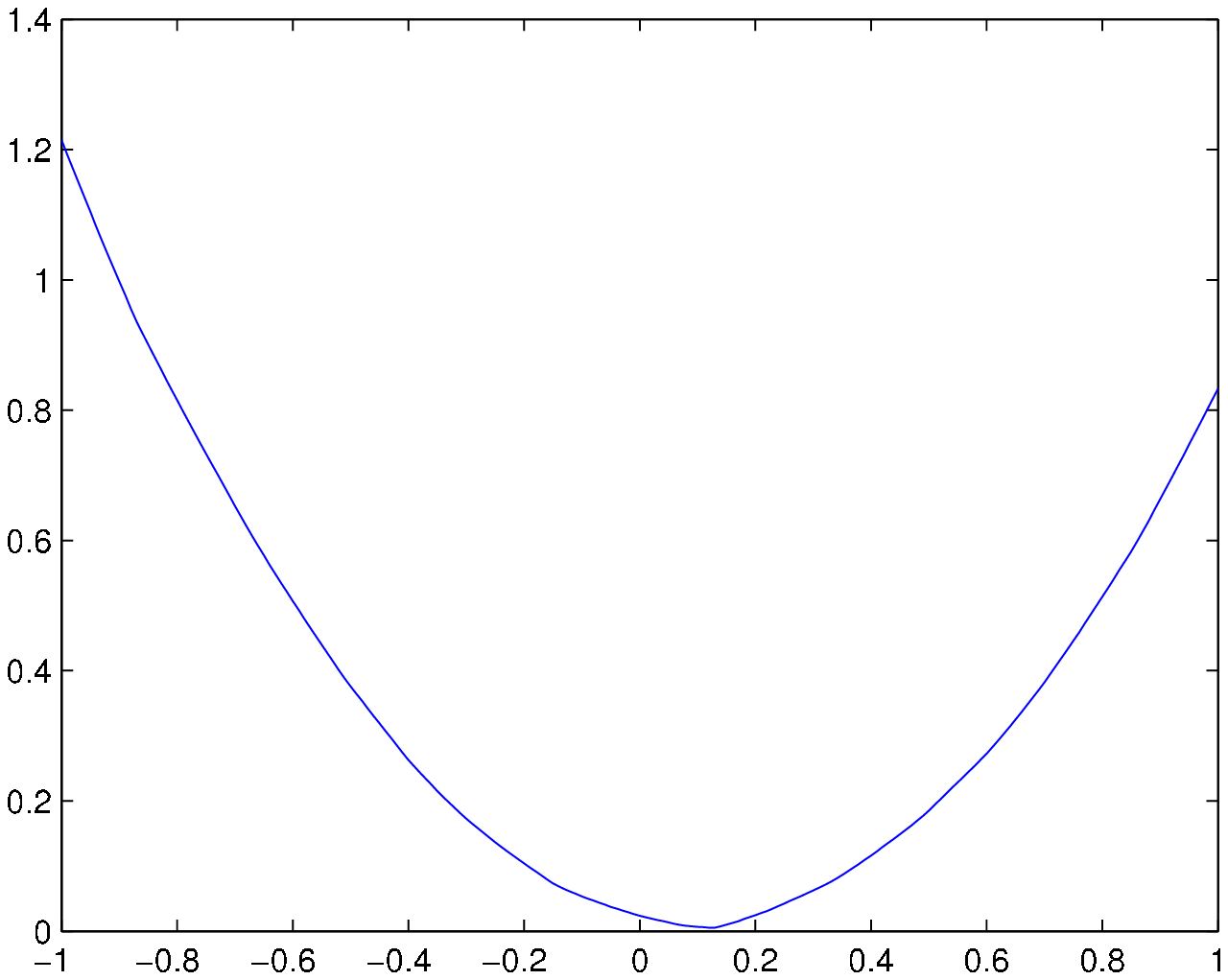}}}

  \centerline{\subfigure[$C_{[F''_0, F''_1]}(\cdot)$ for a
    non-symmetric cost $c(x,y)= |0.5 + x - y|^{1.2} + 0.1\, \cos(2\pi
    x + 1) - 0.3\, \sin(2\pi y -
    0.5)$]{\includegraphics[width=6cm]{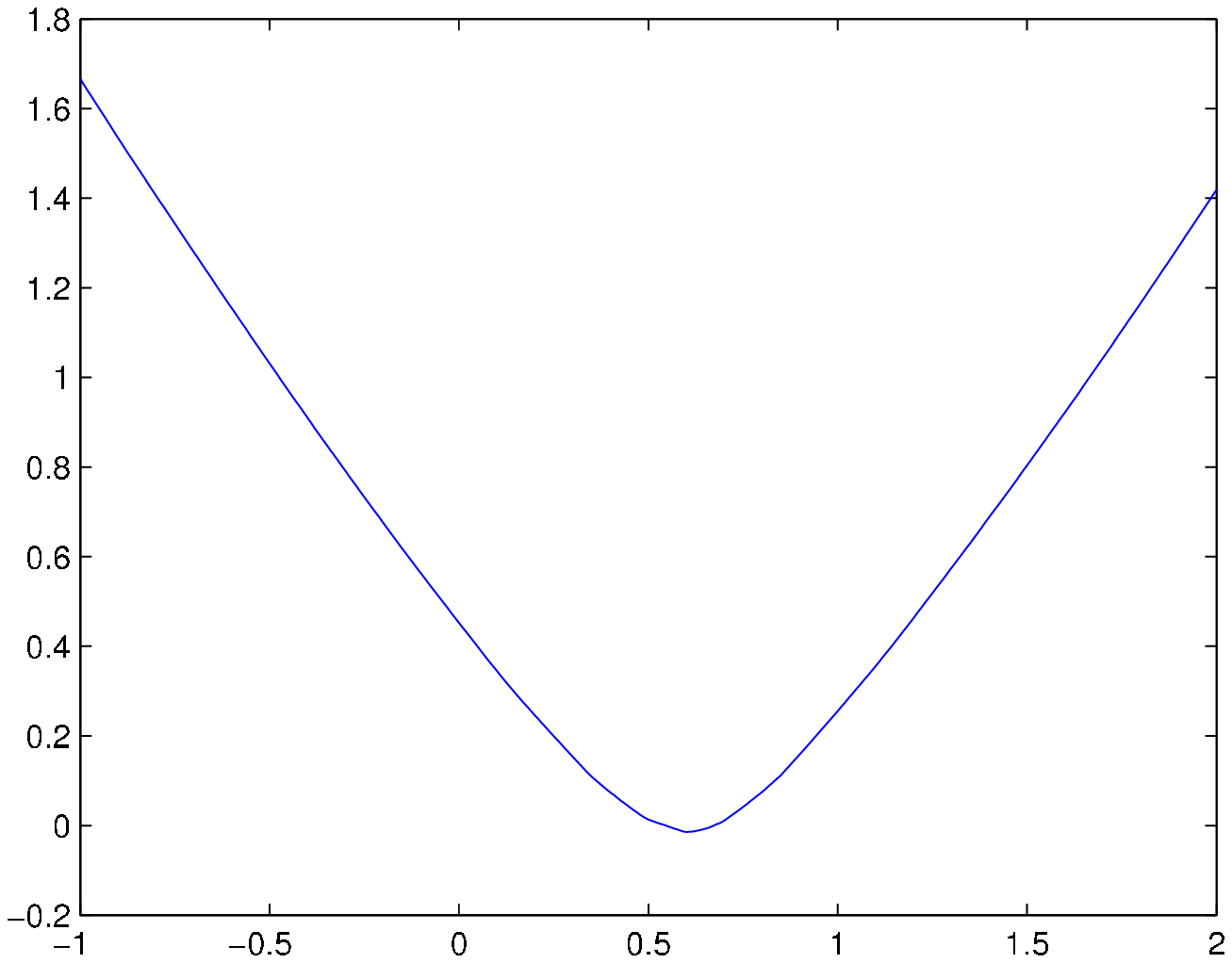}}
 }
  \caption{Average cost functions $C_{[F_0, F_1]}$ for some atomic
    marginals and costs. %
    The cost $|x - y|$ is regarded as $\lim|x - y|^\lambda$, $\lambda
    > 1$, as $\lambda \to 1$.}
\label{fig:cost_function}
\end{figure}

The following technical lemma provides a ``bracket'' for the global
minimum of~$C_{[F_0, F_1]}$ and estimates of its derivatives
independent of $\mu_0$,~$\mu_1$.

\begin{lemma}
  \label{the:convex}
  The average cost $C_{[F_0, F_1]}$ is a convex function that
  satisfies the inequalities
  \begin{equation}
    \label{eq:15}
    \inf_{x, y} c(x, y) \le
    \underline C(\theta) \le C_{[F_0, F_1]}(\theta) \le  \overline C(\theta)
  \end{equation}
  with
  \begin{equation}
    \label{eq:16}
    \underline C(\theta)
    = \inf_{\substack{-1 \le u_1\le 2\\\theta - 1\le u_2 \le \theta + 2}}
    c(u_1, u_2),\quad
    \overline C(\theta)
    = \sup_{\substack{-1 \le u_1\le 2\\\theta - 1\le u_2 \le \theta + 2}}
    c(u_1, u_2).
  \end{equation}
  There exist constants $\underline\Theta < \overline\Theta$ and
  $\underline L, \overline L > 0$ such that the global minimum of $C$
  is achieved on the interval $[\underline\Theta, \overline\Theta]$
  and
  \begin{equation}
    \label{eq:17}
    -\underline L \le C_{[F_0, F_1]}'(\underline\Theta - 0) \le 0
    \le C_{[F_0, F_1]}'(\overline\Theta + 0) \le \overline L,
  \end{equation}
  where $C_{[F_0, F_1]}'(\cdot)$ is the derivative of~$C_{[F_0,
    F_1]}$. %
  These constants are independent on $\mu_0$,~$\mu_1$ and are given
  explicitly by formulas~\eqref{eq:18}, \eqref{eq:19} and~\eqref{eq:20}
  below. %
\end{lemma}

The bounds given in the present lemma are rather loose. %
E.g., for $c(x, y) = |x - y|^\alpha$ with $\alpha > 1$, they are
$\underline C(\theta) = \mathop{\mathrm{conv}} \min(|\theta +
3|^\alpha, |\theta - 3|^\alpha)$, $\overline C(\theta) = \max(|\theta
+ 3|^\alpha, |\theta - 3|^\alpha)$, and $-\underline\Theta =
\overline\Theta = 6$. %
For symmetric costs like this one it is often possible to replace
$[\underline\Theta, \overline\Theta]$ by the interval $[-1, 1]$ which
may be tighter. %

\begin{proof}
  To prove convexity of~$C_{[F_0, F_1]}$ it is sufficient to show that
  $C_{[F_0, F_1]}\bigl(\frac 12(\theta' + \theta'')\bigr) \le \frac 12
  \bigl(C_{[F_0, F_1]}(\theta') + C_{[F_0, F_1]}(\theta'')\big)$ for
  all $\theta'$,~$\theta''$. %
  Let $\theta' < \theta''$, denote $\theta = \frac 12(\theta' +
  \theta'')$ and write
  \begin{displaymath}
    \begin{gathered}
      C_{[F_0, F_1]}(\theta) = \int_0^1 c_{[F_0, F_1]}(v, v + \theta)\,\diff v = \int_{\theta
        - \theta'}^{\theta - \theta' + 1}\!\!\!\!\!\!\!\!\!\! c_{[F_0, F_1]}(v', v' + \theta)\,
      \diff v', \\
      C_{[F_0, F_1]}(\theta') = \int_{\theta - \theta'}^{\theta - \theta' + 1}\!\!\!\!\!\!\!\!\!\!
      c_{[F_0, F_1]}(v', v' + \theta')\, \diff v',\quad C_{[F_0, F_1]}(\theta'') = \int_0^1
      c_{[F_0, F_1]}(v, v + \theta'')\, \diff v.
    \end{gathered}
  \end{displaymath}
  Making the change of variables $v' = v + \theta - \theta'$ and
  taking into account that $\theta - \theta' + \theta = 2\theta -
  \theta' = \theta''$, we get
  \begin{multline*}
    2C_{[F_0, F_1]}(\theta) - C_{[F_0, F_1]}(\theta') - C_{[F_0,
      F_1]}(\theta'') \\
    = \int_0^1 \bigl(c_{[F_0, F_1]}(v, v + \theta) + c_{[F_0, F_1]}(v
    + \theta - \theta', v
    + \theta'') \\
    - c_{[F_0, F_1]}(v + \theta - \theta', v + \theta) - c_{[F_0,
      F_1]}(v, v + \theta'')\bigr)\,\diff v.
  \end{multline*}
  Since $v + \theta - \theta' > v$ and $v + \theta'' > v + \theta$,
  the Monge condition for~$c$ implies that the integrand here is
  negative on a set of nonzero measure, yielding the desired
  inequality for the function~$C_{[F_0, F_1]}$. %
  Note that convexity of~$C_{[F_0, F_1]}$ implies its continuity
  because $C_{[F_0, F_1]}$ is finite everywhere. %

  Bounds~\eqref{eq:15} on~$C_{[F_0, F_1]}(\theta)$ follow from
  \eqref{eq:14} with $v = 0$ because $v' - 1 \le F_0^{-1}(v') \le v' +
  1$, $v' + \theta - 1 \le (F_1^\theta)^{-1}(v') \le v' + \theta + 1$,
  and $0 \le v' \le 1$. %
  Furthermore, the growth condition~\eqref{eq:6} implies that
  $\underline C(\theta) \ge P$ as soon as $|\theta| > R(P) + 3$. %
  Indeed, in this case $|u_2 - u_1| \ge |\theta| - 3 \ge R(P)$ and
  right-hand sides of formulas~\eqref{eq:16} are bounded by~$P$ from
  below. %
  Therefore one can set
  \begin{equation}
    \label{eq:18}
    \underline\Theta = \inf\{\theta\colon \underline C(\theta) =
    \min_{\theta'} \overline C(\theta')\} > -\infty,\
    \overline\Theta = \sup\{\theta\colon \underline C(\theta) =
    \min_{\theta'} \overline C(\theta')\} < \infty,
  \end{equation}
  where $\min$ is attained because $C$ is continuous. %
  
  The set $\arg\min_{\theta'} C_{[F_0, F_1]}(\theta')$ lies on the
  segment $[\underline\Theta, \overline\Theta]$. %
  (Indeed, if e.g. $\theta < \underline\Theta$, then $C_{[F_0,
    F_1]}(\theta) \ge \underline C(\theta) > \min_{\theta'} \overline
  C(\theta') \ge \min_{\theta'} C_{[F_0, F_1]}(\theta')$, so $\theta$
  cannot belong to~$\arg\min_{\theta'} C_{[F_0, F_1]}(\theta')$; a
  similar conclusion holds if~$\theta > \overline\Theta$.) %
  It follows that $C_{[F_0, F_1]}'(\underline\Theta - 0) \le 0 \le
  C_{[F_0, F_1]}'(\overline\Theta + 0)$.

  By convexity $C_{[F_0, F_1]}'(\overline\Theta + 0) \le (C_{[F_0, F_1]}(\theta) -
  C_{[F_0, F_1]}(\overline\Theta))/(\theta - \overline\Theta)$ for all
  $\theta\ge\overline\Theta$. %
  The right-hand side of the latter inequality can be estimated from
  above by
  \begin{equation}
    \label{eq:19}
    \overline L = \inf_{\theta\ge\overline\Theta}
    \frac{\overline C(\theta) - \underline C(\overline\Theta)}
    {\theta - \overline\Theta}.
  \end{equation}
  The ratio in the right-hand side takes finite values, so $\overline
  L$ is finite. %
  This establishes the inequality $C_{[F_0, F_1]}'(\overline\Theta +
  0) \le \overline L$. %
  The rest of~\eqref{eq:17} is given by a symmetrical argument; in
  particular
  \begin{equation}
    \label{eq:20}
    \underline L = \inf_{\theta \le \underline\Theta}
    \frac{\overline C(\theta) - \underline C(\underline\Theta)}
    {\underline\Theta - \theta}.
  \end{equation}
\end{proof}

\begin{definition}
  A locally optimal transport plan~$\gamma_{\theta_0}$ is called
  \emph{globally optimal} if $\theta_0 \in \arg\min_\theta C_{[F_0, F_1]}(\theta)$.
\end{definition}

We can now reduce minimization of~\eqref{eq:1} on the unit circle to
minimization of~\eqref{eq:2} on~$\rset$, which involves the cost
function $c$ rather than~$\hat c$:

\begin{theorem}
  The canonical projection $\pi\colon \rset\to\tset$ establishes a
  bijection between globally optimal transport plans on
  $\rset\times\rset$ and transport plans on $\tset\times\tset$ that
  minimize~\eqref{eq:1}.
\end{theorem}

\begin{proof}
  A transport plan $\gamma$ on~$\tset\times\tset$
  minimizes~\eqref{eq:1} if it is a projection of a transport plan
  on~$\rset\times\rset$ that locally minimizes the transport cost
  defined by the cost function $\hat c(x, y) = \min_{k\in\zset} c(x, y
  + k)$ (see introduction; $\min$ here is attained because of the integer shift
  invariance and growth conditions \eqref{eq:5},~\eqref{eq:6}). %

  Denote $S = \{(x, y)\colon c(x, y) = \hat c(x, y)\}$ and observe
  that the support of the globally optimal plan~$\gamma_{\theta_0}$
  lies within~$S$: indeed, if it did not, there would exist a
  (nonlocal but periodic) modification of~$\gamma_{\theta_0}$ bringing
  some of the mass of each period to~$S$ and thus reducing the average
  cost. %
  Therefore $\gamma_{\theta_0}$ is locally optimal with respect to the
  cost~$\hat c(x, y)$ and its projection to~$\tset\times\tset$
  minimizes~\eqref{eq:1}.

  Conversely, a minimizing transport plan on $\tset\times\tset$ can be
  lifted to $\rset\times\rset$ in such a way that its support lies
  inside~$S$ (translations of arbitrary pieces of support by integer
  increments along $x$ and~$y$ axes are allowed because they leave
  $\hat c(x, y)$ invariant). %
  Therefore its average cost per period cannot be less than that of a
  globally optimal transport plan on~$\rset\times\rset$.
\end{proof}

\section{Fast global transport optimization}
\label{sec:algor-transp-optim}

In a typical application, such as the image processing problem
described in the introduction, measures $\mu_0$ and~$\mu_1$ come in
the form of \emph{histograms}, i.e., discrete distributions supported
on subsets $\hat X = \{\hat x_1, \hat x_2, \dots, \hat x_{n_0}\}$ and
$\hat Y = \{\hat y_1, \hat y_2, \dots, \hat y_{n_1}\}$ of the unit
circle. %
These two sets may coincide. %
In what follows we replace $\hat X$ and~$\hat Y$ with their lifts to
the universal cover $X$ and~$Y$ and assume that the points of the
latter pair of sets are sorted and numbered in an increasing order:
\begin{displaymath}
  \begin{gathered}
    \dots < x_{-1} < x_0 = x_{n_0} - 1 \le 0 < x_1 < \dots < x_{n_0}
    \le 1 < x_{n_0 + 1} = x_1 + 1 < \dots,\\
    \dots < y_{-1} < y_0 = y_{n_1} - 1 \le 0 < y_1 < \dots < y_{n_1}
    \le 1 < y_{n_1 + 1} = y_1 + 1 < \dots.
  \end{gathered}
\end{displaymath}
Denote masses of these points by $\mu_0(\{x_i\}) = m^{(0)}_i$,
$\mu_1(\{y_j\}) = m^{(1)}_j$; these are assumed to be arbitrary positive
real numbers satisfying $\sum_{1\le i\le n_0} m^{(0)}_i = \sum_{1\le j\le n_1}
m^{(1)}_j = 1$. %

\subsection{Computation of the average cost and its derivative}
\label{sec:comp-aver-cost}

Define $j(\theta)$ as the index of $\min \{y_j\colon F_1^\theta(y_j) >
0\}$ and denote $y_1^\theta = y_{j(\theta)}$, $y_2^\theta =
y_{j(\theta) + 1}$, \dots, $y_{n_1}^\theta = y_{j(\theta) + n_1 - 1}$.
All the values
\begin{equation}
  \label{eq:21}
  F_0(x_1), F_0(x_2), \dots, F_0(x_{n_0}), F_1^\theta(y_1^\theta),
  F_1^\theta(y_2^\theta), \dots, F_1^\theta(y_{n_1}^\theta)
\end{equation}
belong to the segment $(0, 1]$. %
We now sort these values into an increasing sequence, denote its
elements by $v_{(1)} \le v_{(2)} \le \dots \le v_{(n_0 + n_1)}$ and
set $v_{(0)} = 0$. %
Note that for each $v$ such that $v_{(k - 1)} < v < v_{(k)}$ with
$1\le k\le n_0 + n_1$ the values $x_{(k)} = F_0^{-1}(v)$ and $y_{(k)}
= (F_1^\theta)^{-1}(v)$ are uniquely defined and belong to $X$,~$Y$. %
It is now easy to write an expression for the function $C_{[F_0, F_1]}$:
\begin{equation}
  \label{eq:22}
  C_{[F_0, F_1]}(\theta) = \sum_{1\le k\le n_0 + n_1}
  c(x_{(k)}, y_{(k)})\, (v_{(k)} - v_{(k - 1)}).
\end{equation}

Observe that, as the parameter $\theta$ increases by~$\Delta\theta$,
those $v_{(k)}$ that correspond to values $F_1^\theta$ decrease by the
same increment. %
Let $F^\theta_1(y_{j_0})$ be such a value. %
As it appears in~\eqref{eq:22} twice, first as $v_{(k)}$ and then as
$-v_{(k - 1)}$ in the next term of the sum, it will make two
contributions to the derivative $C_{[F_0, F_1]}'(\theta)$:
$-c(F_0^{-1}(F_1^\theta(y_{j_0})), y_{j_0})$ and
$c(F_0^{-1}(F_1^\theta(y_{j_0})), y_{j_0 + 1})$ (see
fig.~\ref{fig:deriv}, top). %

\begin{figure}[htp]
  \centerline{\subfigure[$\theta$ not exceptional]%
    {\begin{tikzpicture}
        \draw[->,>=stealth] (0, 0) -- (0, 3) node [left] {$v$}; %
        \draw[->,>=stealth] (-.5, 0) -- (5.5, 0) node [above] {$u$}; %
        \draw (0, 0) node [above left] {$O$}; %
        \begin{scope}[very thick, xshift=1.25cm, yshift=2.6cm]
          \draw (0,-1.75) -- (0, 0) node [left] {$F_0$} -- (1.5, 0); %
          \draw[dotted] (0, -2) -- (0, -1.75); %
          \draw[dotted] (1.5, 0) -- (1.75, 0); %
          \draw[thin, dotted] (0, -1.75) -- (0, -2.6) node [below]
          {$F_0^{\smash{-1}}(F_1^{\smash\theta}(y_j)) \equiv X$ };
        \end{scope}
        \begin{scope}[very thick, xshift=3.5cm, yshift=1.75cm]
          \draw(0, -1) -- (0, 0) -- (1.5, 0) node [right]
          {$F_1^\theta$} -- (1.5, .35); %
          \draw[dotted] (0, -1) -- (0, -1.25); %
          \draw[dotted] (1.5, .35) -- (1.5, .6); %
          \draw[densely dashed] (0, -.5) -- (1.5, -.5) node [right]
          {$F_1^{\theta + \Delta\theta}$} -- (1.5, 0); %
          \draw[thin, dotted] (0, -1.25) -- (0, -1.75) node [below]
          {$\mathstrut y_j$}; %
          \draw[thin, dotted] (1.5, 0) -- (1.5, -1.75) node [below]
          {$\mathstrut y_{j + 1}$}; %
          \draw[thin, dotted] (0, 0) -- (-3.5, 0); %
          \draw[thin, dotted] (0, -.5) -- (-3.5, -.5); %
          \draw (-3.5, -.25) node [left] {$\Delta\theta$};
        \end{scope}
        \draw[<->] (1.3, 1) -- (3.45, 1); %
        \draw (2.4, 1) node [below] {$c(X, y_j)$}; %
        \draw[<->] (1.3, 2) -- (4.95, 2); %
        \draw (3.2, 2) node [above] {$c(X, y_{j + 1})$};
      \end{tikzpicture}}}

  \centerline{\subfigure[$\theta$ exceptional, case $C'_{[F_0,
      F_1]}(\theta - 0)$]%
    {\begin{tikzpicture}
        \draw[->,>=stealth] (0, 0) -- (0, 3) node [left] {$v$}; %
        \draw[->,>=stealth] (-.5, 0) -- (5.5, 0) node [above] {$u$}; %
        \draw (0, 0) node [above left] {$O$}; %
        \begin{scope}[very thick, xshift=.75cm, yshift=2cm]
          \draw (0, -1.25) -- (0, 0) node [below left] {$F_0$} -- (1,
          0) -- (1, .75); %
          \draw[dotted] (0, -1.25) -- (0, -1.5); %
          \draw[dotted] (1, .75) -- (1, 1); %
          \draw[thin, dotted] (1, 0) -- (1, -2) node [below]
          {$F_0^{\smash{-1}}(F_1^{\smash\theta}(y_j))$};
        \end{scope}
        \begin{scope}[very thick, xshift=3.25cm, yshift=2cm]
          \draw (0, -1) -- (0, 0) -- (.75, 0) node [right]
          {$F_1^\theta$} -- (.75, .75); %
          \draw[dotted] (0, -1) -- (0, -1.25); %
          \draw[dotted] (.75, .75) -- (.75, 1); %
          \draw[densely dashed] (0, 0) -- (0, .5) -- (.75, .5) node
          [right] {$F_1^{\theta - \Delta\theta}$}; %
          \draw[thin, dotted] (0, -1.25) -- (0, -2) node [below]
          {$\mathstrut y_j$}; %
          \draw[thin, dotted] (.75, 0) -- (.75, -2) node [below]
          {$\mathstrut y_{j + 1}$}; %
          \draw[thin, dotted] (.75, .5) -- (-3.25, .5); \draw[thin,
          dashed] (.75, 0) -- (-3.25, 0); \draw (-3.25, .25) node
          [left] {$\Delta\theta$};
        \end{scope}
        \draw[<->] (1.8, 2.25) -- (3.95, 2.25); %
        \draw[<->] (1.8, 1.75) -- (3.2, 1.75); %
      \end{tikzpicture}} %
    \subfigure[$\theta$ exceptional, case $C'_{[F_0, F_1]}(\theta +
    0)$]%
    {\begin{tikzpicture}
        \draw[->,>=stealth] (0, 0) -- (0, 3) node [left] {$v$}; %
        \draw[->,>=stealth] (-.5, 0) -- (5.5, 0) node [above] {$u$}; %
        \draw (0, 0) node [above left] {$O$}; %
        \begin{scope}[very thick, xshift=.75cm, yshift=2cm]
          \draw (0, -1.25) -- (0, 0) node [above left] {$F_0$} -- (1,
          0) -- (1, .75); %
          \draw[dotted] (0, -1.25) -- (0, -1.5); %
          \draw[dotted] (1, .75) -- (1, 1); \draw[thin, dotted] (0, 0)
          -- (0, -2) node [below]
          {$F_0^{\smash{-1}}(F_1^{\smash\theta}(y_j) - 0)$}; %
        \end{scope}
        \begin{scope}[very thick, xshift=3.25cm, yshift=2cm]
          \draw (0, -1) -- (0, 0) -- (.75, 0) node [right]
          {$F_1^\theta$} -- (.75, .75); %
          \draw[dotted] (0, -1) -- (0, -1.25); %
          \draw[dotted] (.75, .75) -- (.75, 1); %
          \draw[densely dashed] (0, -.5) -- (.75, -.5) node [right]
          {$F_1^{\theta + \Delta\theta}$} -- (.75, 0); %
          \draw[thin, dotted] (0, -1.25) -- (0, -2) node [below]
          {$\mathstrut y_j$}; %
          \draw[thin, dotted] (.75, 0) -- (.75, -2) node [below]
          {$\mathstrut y_{j + 1}$}; %
          \draw[thin, dashed] (0, 0) -- (-3.25, 0); \draw[thin,
          dotted] (0, -.5) -- (-3.25, -.5); \draw (-3.25, -.25) node
          [left] {$\Delta\theta$};
        \end{scope}
        \draw[<->] (.8, 1.75) -- (3.95, 1.75); %
        \draw[<->] (.8, 1.25) -- (3.2, 1.25); %
      \end{tikzpicture}}}

  \caption{Derivation of expressions
    \protect\eqref{eq:23},~\protect\eqref{eq:24} for~$C_{[F_0,
      F_1]}'(\theta\pm 0)$. %
    Thick lines show fragments of complete graphs of
    $F_0$,~$F_1^\theta$ corresponding to $j$th terms in
    \protect\eqref{eq:23},~\protect\eqref{eq:24}, thin dashed line
    (bottom) marks the common value of $F_0$ and $F_1^\theta$. %
    Note that $X = F_0^{-1}(F^\theta_1(y_j))$ is equivalently
    expressed as $\inf\{x\colon F_0(x) > F^\theta_1(y_j)\}$. %
  }
  \label{fig:deriv}
\end{figure}
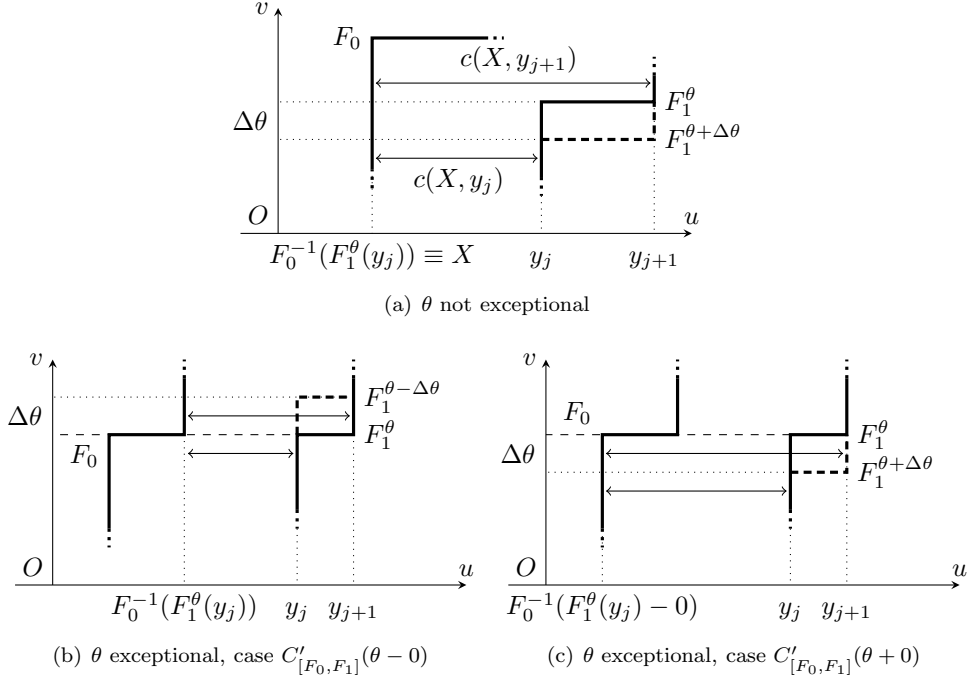

Moreover, there are exceptional values of~$\theta$ for which two of
the values in~\eqref{eq:21} coincide and their ordering in the
sequence $(v_{(k)})$ changes. %
For such values of~$\theta$ the derivative $C_{[F_0, F_1]}'$ has
different right and left limits, as illustrated in
fig.~\ref{fig:deriv}, bottom:
\begin{gather}
  \label{eq:23}
  C_{[F_0, F_1]}'(\theta - 0) = \!\!\!\sum_{1\le j\le n_1}\!\!\!
  \bigl(c(F_0^{-1}(F^\theta_1(y_j)), y_{j + 1})
  - c(F_0^{-1}(F^\theta_1(y_j)), y_j)\bigr), \\
  \label{eq:24}
  C_{[F_0, F_1]}'(\theta + 0) = \!\!\!\!\!\sum_{1\le j\le n_1}\!\!\!\!
  \bigl(c(F_0^{-1}(F^\theta_1(y_j) - 0), y_{j + 1})
  - c(F_0^{-1}(F^\theta_1(y_j) - 0), y_j)\bigr).
\end{gather}
If $\theta$ is not exceptional, the value of~$C_{[F_0, F_1]}'(\theta)$
is given by the first of these formulas. %

The function $C_{[F_0, F_1]}$ is therefore piecewise affine (see in
particular fig.~\ref{fig:cost_function}, where this function is
plotted for atomic marginals $\mu_0$,~$\mu_1$). %
Moreover, from the Monge condition~\eqref{eq:4} it follows that
$C_{[F_0, F_1]}'(\theta - 0) < C_{[F_0, F_1]}'(\theta + 0)$ at
exceptional points, giving an alternative proof of convexity
of~$C_{[F_0, F_1]}(\theta)$ in the discrete case. %

\begin{lemma}
  \label{lm:deriv-opcount}
  Values of $C$ and its left and right derivatives can be computed for
  any~$\theta$ using at most $O(n_0 + n_1)$ comparisons and
  evaluations of~$c(x, y)$.
\end{lemma}

\begin{proof}
  Sorting the $n_0 + n_1$ values~\eqref{eq:21} into an increasing
  sequence requires $n_0 + n_1 - 1$ comparisons (one starts with
  comparing $F_0(x_1)$ and $F^\theta_1(y^\theta_1)$ to determine
  $v_{(1)}$, and after this each of the remaining values is considered
  once until there remains only one value, which is assigned
  to~$v_{(n_0 + n_1)}$ with no further comparison). %
  At the same time, pointers to $x_{(k)}$ and~$y_{(k)}$ should be
  stored. %
  After this preliminary stage, to find the values for $C_{[F_0,
    F_1]}$ and its one-sided derivatives it suffices to evaluate each
  of the $n_0 + n_1$ terms in~\eqref{eq:22} and to take into account
  the corresponding contribution of plus or minus $c(x_{(k)}, y_{(k)})$
  to the value of~$C_{[F_0, F_1]}'(\theta)$, paying attention to
  whether the value of~$\theta$ is exceptional or not. %
  All this can again be done in $O(n_0 + n_1)$ operations.
\end{proof}

\subsection{Transport optimization algorithm}
\label{sec:algorithm}

Fix $\epsilon > 0$ and set $L = \max\{\underline L, \overline
L\}$. %
Recall that $\underline L$,~$\overline L$, as well as the parameters
$\underline\Theta$,~$\overline\Theta$ that are used in the algorihtm
below, are defined by explicit formulas in Lemma~\ref{the:convex} and
do not depend on measures $\mu_0$,~$\mu_1$. %
The minimum of~$C_{[F_0, F_1]}(\theta)$ can be found to accuracy
$\epsilon$ using the following binary search technique:

\bigskip

\begin{enumerate}
\item\label{item:3} Initially set $\underline\theta :=
  \underline\Theta$ and $\overline\theta := \overline\Theta$, where
  $\underline\Theta$,~$\overline\Theta$ are defined in
  Lemma~\ref{the:convex}.
\item\label{item:4} Set $\theta := \frac 12(\underline\theta +
  \overline\theta)$.
\item\label{item:5} Compute $C_{[F_0, F_1]}'(\theta - 0)$, $C_{[F_0, F_1]}'(\theta + 0)$.
\item If $C_{[F_0, F_1]}'(\theta - 0)\le 0\le C_{[F_0, F_1]}'(\theta + 0)$, then $\theta$ is the
  required minimum; stop.
\item\label{item:6} If $\overline\theta - \underline\theta <
  \epsilon/L$, then compute $C_{[F_0, F_1]}(\underline\theta)$,
  $C_{[F_0, F_1]}(\overline\theta)$, solve the linear
  equation
  \begin{equation}
    \label{eq:25}
    C_{[F_0, F_1]}(\underline\theta) +  C_{[F_0, F_1]}'(\underline\theta + 0)(\theta -
    \underline\theta) =
    C_{[F_0, F_1]}(\overline\theta) + C_{[F_0, F_1]}'(\overline\theta - 0)(\theta -
    \overline\theta)
  \end{equation}
  for $\theta$, and stop.
\item\label{item:8} Otherwise set $\underline\theta := \theta$ if
  $C_{[F_0, F_1]}'(\theta + 0) < 0$, or $\overline\theta := \theta$ if $C_{[F_0, F_1]}'(\theta -
  0) > 0$.
\item Go to step~\ref{item:4}.
\end{enumerate}

\bigskip

It follows from inequalities~\eqref{eq:17} of Lemma~\ref{the:convex}
that the minimizing value of~$\theta$ belongs to the segment
$[\underline\Theta, \overline\Theta]$. %
Therefore at all steps
\begin{equation}
  \label{eq:26}
  C_{[F_0, F_1]}'(\underline\theta + 0) \le 0 \le C_{[F_0, F_1]}'(\overline\theta - 0)
\end{equation}
and the segment $[\underline\theta, \overline\theta]$ contains the
minimum of~$C$. %

Step~\ref{item:6} requires some comments. %
By convexity, $-\underline L\le C_{[F_0, F_1]}'(\theta \pm 0) \le
\overline L$ for all $\underline\Theta \le \theta \le
\overline\Theta$, i.e., $|C_{[F_0, F_1]}'(\theta\pm 0)| \le L$ at all
steps. %
When $\overline\theta - \underline\theta < \epsilon/L$, this bound
ensures that for any $\theta'$ in $[\underline\theta,
\overline\theta]$ the minimal value of~$C$ is within $\epsilon/L \cdot
L = \epsilon$ from $C_{[F_0, F_1]}(\theta')$. %
If there is a single exceptional value of~$\theta$ in that interval,
then it is located precisely at the solution of~\eqref{eq:25} and must
be a minimum of~$C$ because of~\eqref{eq:26}, so the final value
of~$\theta$ is the exact solution; otherwise it is an approximation
with guaranteed accuracy. %

The final value of~$\theta$ will certainly be exact when masses of all
atoms are rational numbers having the least common denominator~$M$ and
$\epsilon < 1/M$. %
Indeed, in this case any interval $[\underline\theta,
\overline\theta]$ of length $\epsilon$ can contain at most one
exceptional value of~$\theta$.

Since at each iteration the interval $[\underline\theta,
\overline\theta]$ is halved, step~\ref{item:6} will be achieved in
$O(\log_2((\overline\Theta - \underline\Theta)/(\epsilon/L)))$
iterations. %
By Lemma~\ref{lm:deriv-opcount} each instance of step~\ref{item:5}
(and equation~\eqref{eq:25}) takes $O(n_0 + n_1)$ operations. %
Thus we obtain the following result. 

\begin{theorem}
  \label{thm:complexity}
  The above binary search algorithm takes $O((n_0 +
  n_1)\log(1/\epsilon))$ comparisons and evaluations of $c(x, y)$ to
  terminate. %
  The final value of~$\theta$ is within $\epsilon/L$ from the global
  minimum, and $C_{[F_0, F_1]}(\theta) \le \min_\theta C_{[F_0,
    F_1]}(\theta) + \epsilon$. %
  When all masses $m^{(0)}_i$,~$m^{(1)}_j$ are rational with the least
  common denominator~$M$, initializing the algorithm with $\epsilon =
  1/2M$ leads to an exact solution in $O((n_0 + n_1)\log M)$
  operations.
\end{theorem}

\subsection{Experiments}
\label{sec:experiments}

\begin{figure}[htp]
  \centerline{\subfigure[Average computing time, sec., vs ${n_0 +
      n_1}$ for $\epsilon = 10^{-10}$.]%
    {\includegraphics[width=.48\textwidth]{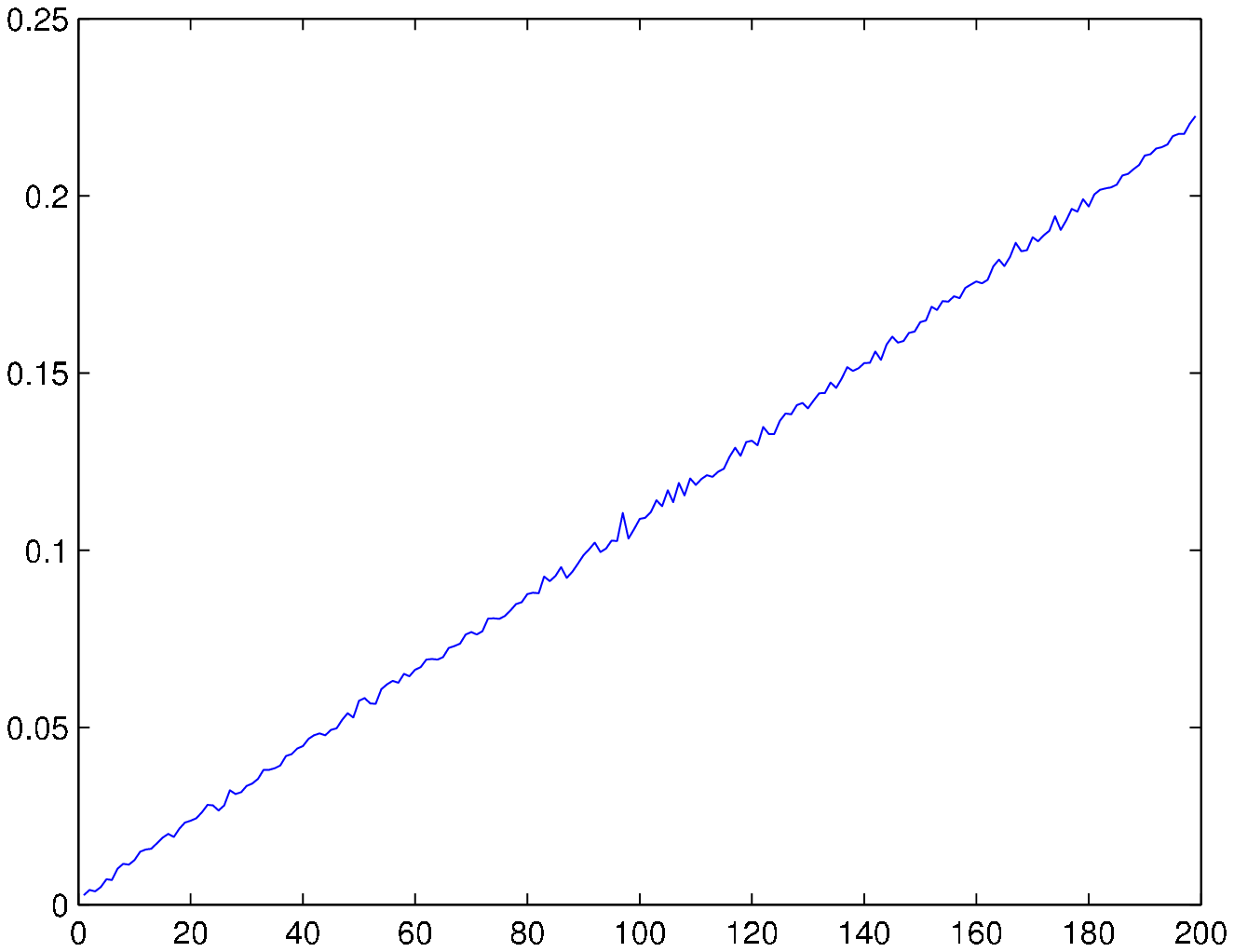}}
    \quad %
    \subfigure[Average computing time, sec., vs $\log_{10} \epsilon$
    for $n_0=n_1=10$.]%
    {\includegraphics[width=.48\textwidth]{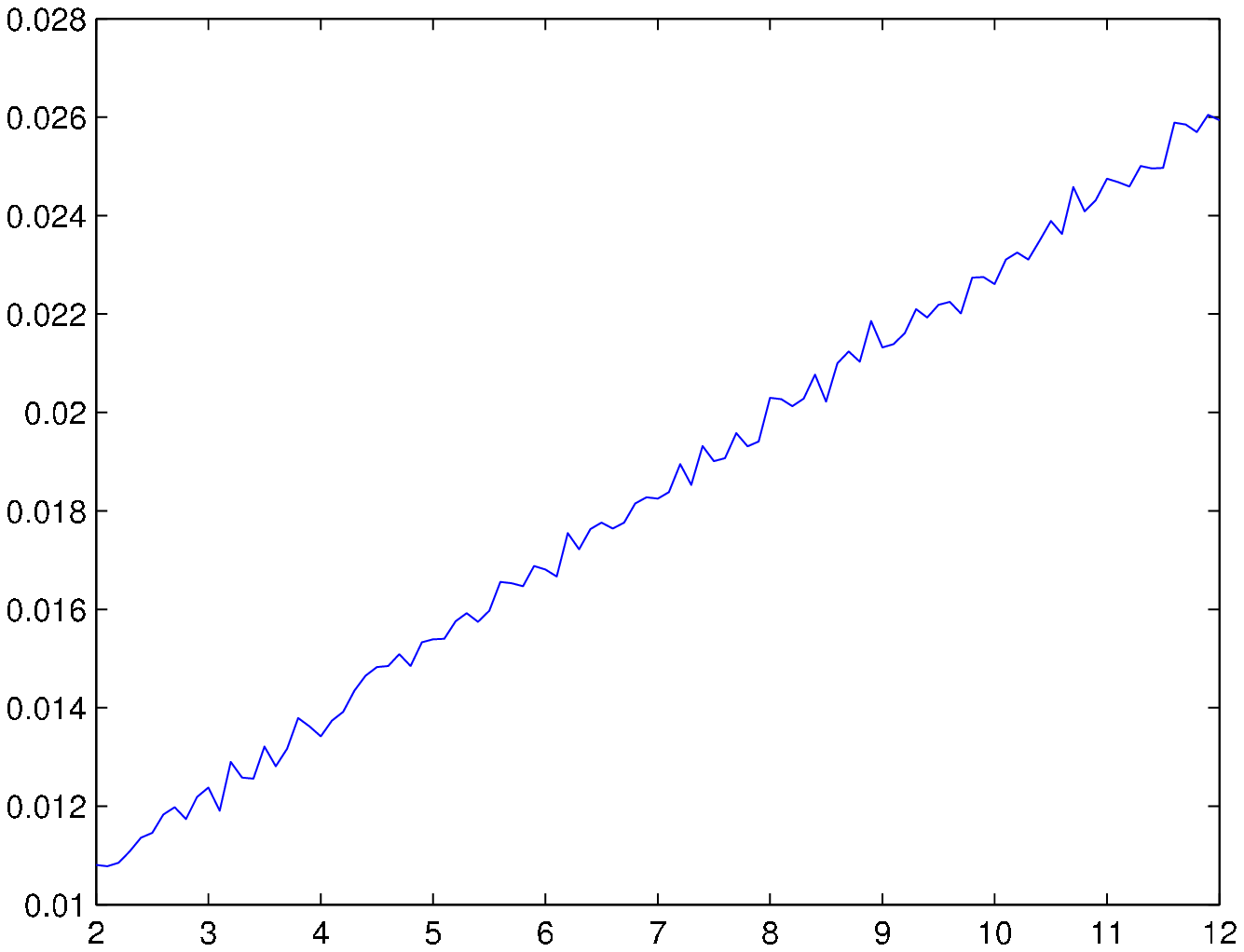}}}
  \caption{Average computing time of the algorithm for different
    values of $(n_0+n_1)$ and $\log_{10}\epsilon$. %
    The experiment was performed on a PC with a $3.00\,\text{GHz}$
    processor.}
  \label{fig:complexity}
\end{figure}

\begin{figure}[htp]
  \centering \subfigure[Top: F. Maliavin, \textit{Whirlwind} (1916).\newline %
  (b) Right: P. Puvis de Chavanne, \textit{Jeunes filles au bord de la
    mer} (1879).]%
  {\vtop{\vbox to
      0pt{\null}\hbox{\includegraphics[width=0.5\textwidth]{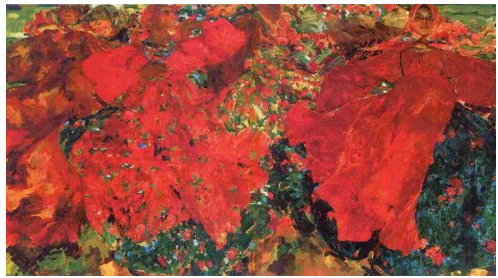}}\bigskip}}\quad
  \subfigure%
  {\vtop{\vbox to 0pt{\null}\hbox{\includegraphics[width=0.28\textwidth]{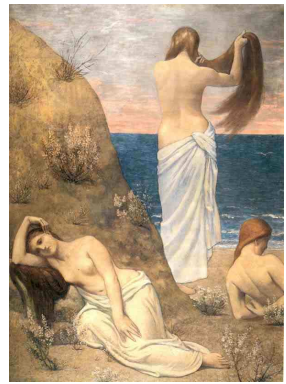}}}}\\
  \subfigure[Puvis' hues transported to Maliavin's canvas.]%
  {\includegraphics[width=.8\textwidth]{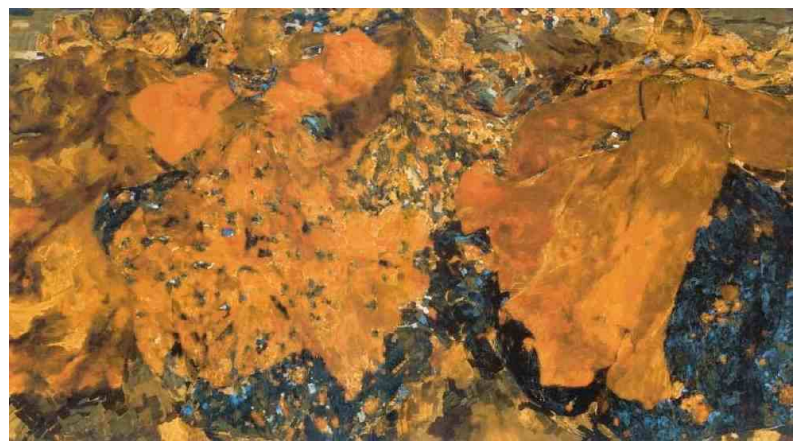}}
  \caption{Optimal matching of the hue component of color.}
  \label{fig:fun}
\end{figure}

We tested experimentally the estimates of Theorem~\ref{thm:complexity}
for time complexity as a function of parameters of the problem. %
The average computing time of the algorithm for different values of
$n_0 + n_1$ and $\epsilon$ is illustrated in
fig.~\ref{fig:complexity}. %
These results have been obtained using the following procedure. %
For each value of $n_0$ and $n_1$, points $\{\hat x_1, \hat x_2,
\dots, \hat x_{n_0}\}$ and $\{\hat y_1, \hat y_2, \dots, \hat
y_{n_1}\}$, which constitute the support of distributions $\mu_0$
and~$\mu_1$, are drawn independently from the uniform distribution on
$[0, 1]$ and sorted. %
The masses $m^{(0)}_i$,~$i=1 \dots n_0$ and $m^{(1)}_j$,~$j=1\dots
n_1$ are then drawn from the uniform distribution and normalized such
that $\sum_{1 \le i\le n_0} m^{(0)}_i = \sum_{1 \le j\le n_1}
m^{(1)}_j = 1$. %
Finally the transport cost is minimized for $c(x, y) = |x - y|$. %
The code used to produce this figure is available online at the web
site of the OTARIE project
\url{http://www.mccme.ru/~ansobol/otarie/software.html}. %

In the first experiment the value of $\epsilon$ was set to~$10^{-10}$,
the algorithm was run $10$~times for each pair $(n_0, n_1)$ with $1
\le n_0, n_1 \le 100$, and the computing times were averaged. %
In the second experiment $(n_0, n_1)$ was fixed at $(10, 10)$ and the
average computing time was similarly computed for different values
of~$\epsilon$. %
The averaged computing times for the two experiments are plotted in
fig.~\ref{fig:complexity}. %
Observe the manifest linear dependence of computing time on $n_0 +
n_1$ and~$\log\epsilon$.

The next figure is a concrete, if not entirely serious, illustration
of optimal matching in the case of distributions on the ``color
circle.'' %

Recall that in the HSL (Hue, Saturation, and Lightness) color model,
the color space is represented in cylindrical coordinates. %
The polar angle corresponds to the \emph{hue}, or the degree to which
a color can be described as similar to or different from other colors
(as opposed to difference in saturation or lightness between shades of
the same color). %

We chose two famous paintings, one Russian and one French, whose
highly different coloring is characteristic of the two painters, the
Expressionist Filipp Maliavin (1869--1940) and the Symbolist Pierre
Puvis de Chavannes (1824--1898). %
An optimal matching of the hue distributions according to the linear
cost $c(x, y) = |x - y|^2$ was used to substitute hues of the first
painting with the corresponding hues of the second one while
preserving the original values of saturation and brightness. %
In spite of the drastic change in coloring, the optimality of matching
ensures that warm and cold colors retain their quality and the overall
change of aspect does not feel arbitrary or artificial. %

\section{Related algorithmic work}
\label{sec:disc-concl}

Fast algorithms for the transportation problem on the circle, with the
Euclidean distance $|x - y|$ as a cost, have been proposed in a number
of works. %
Karp and Li \cite{Karp.R:1975} consider an \emph{unbalanced} matching,
where the total mass of the two histograms are not equal and elements
of the smaller mass have to be optimally matched to a subset of
elements of the larger mass. %
A balanced optimal matching problem has later been considered
independently by Werman et al \cite{Werman.M:1986}; clearly, the
balanced problem can always be treated as a particular case of the
unbalanced one. %
In both of these works $O(n\log n)$ algorithms are obtained for the
case where all points have unit mass. %

Aggarwal et al \cite{Aggarwal.A:1992} present an algorithm improving
Karp and Li's results for an unbalanced transportation problem on the
circle with general integer weights and the same cost function~$|x -
y|$. %
They also consider a general cost function $c(x, y)$ that satisfies
the Monge condition and an additional condition of \emph{bitonicity}:
for each $x$, the function $c(x, y)$ is nonincreasing in~$y$ for
$y<y_0(x)$ and nondecreasing in~$y$ for $y > y_0(x)$. %
Note that this rules out the circular case. %
The second algorithm of~\cite{Aggarwal.A:1992} is designed for bitonic
Monge costs and runs in $O(n\log M)$ time for an unbalanced
transportation problem with integer weights on the line, where $M$ is
the total weight of the matched mass and $n$ is the number of points
in the larger histogram. %

The algorithm proposed in the present article only applies to the
\emph{balanced} problem for a Monge cost. %
However it does not involve bitonicity and is therefore applicable on
the circle, where it achieves the same $O(n\log M)$ time as the second
algorithm of \cite{Aggarwal.A:1992} if all weights are integer
multiples of~$1/M$. %
Although our theory is developed for the case of costs satisfying a
strict inequality in the Monge condition, it can be checked that the
discrete algorithm works for the case $c(x, y) = |x - y|$, which can
be treated as a limit of $|x - y|^\lambda$, $\lambda > 1$, as $\lambda
\to 1$ \cite{Rabin.J:2009a}. %

Finally we note that results of \cite{Aggarwal.A:1992} were extended
in a different direction by McCann \cite{McCann.R:1999}, who provides,
again in the balanced setting, a generalization of their first
algorithm to the case of a general cost of the \emph{concave} type on
the open line. %
This case is opposite to Monge costs and requires completely different
tools. %
Indeed, for a strictly concave cost such as $c(x, y) = \sqrt{|x - y|}$
the notion of locally optimal transport plan on the universal cover
does not make sense: concave costs favor long-haul transport over
local rearrangements, destroying local finiteness. %


\end{document}